\documentclass[11pt,twoside,english]{article}
\usepackage{mathrsfs}
\usepackage{amsmath, amsfonts, amssymb, amsthm}
\usepackage[T1]{fontenc}
\usepackage[english]{babel}
\usepackage{graphicx}
\usepackage{hyperref}
\hypersetup{hidelinks}
\usepackage{algorithm}
\usepackage{algorithmic}
\usepackage{booktabs}
\usepackage{enumitem}
\usepackage{mathtools}
\usepackage[title]{appendix}
\usepackage{authblk}

\newtheorem{theorem}{Theorem}

\newtheorem{proposition}[theorem]{Proposition}

\newtheorem{definition}[theorem]{Definition}
\newtheorem{corollary}[theorem]{Corollary}
\newtheorem{remark}[theorem]{Remark}
\newtheorem{assumption}[theorem]{Assumption}

\newcommand{\cP}{\ensuremath{\mathcal P}}

\allowdisplaybreaks[3]

\begin{document}
\title{Continuous-Time Reinforcement Learning for $N$-Player Stochastic Differential Games with Exploratory Policies}
\author[1]{Jisheng Liu\thanks{Email: \texttt{jsliu23@m.fudan.edu.cn}}}
\affil[1]{School of Mathematical Sciences, Fudan University, Shanghai, China}
\date{}
\maketitle

\begin{abstract}
We study entropy-regularized exploratory control in finite $N$-player stochastic differential games under a response model in which each player conditions on the opponents' currently realized actions and evaluates continuation with that profile frozen. The resulting Gibbs best responses form a system of full conditional densities, which need not admit a common joint law. We characterize joint realizability by a cross-partial condition on the entropy-scaled Hamiltonian gradients and, on simply connected action domains, by an equivalent entropy-weighted potential structure. When compatibility fails, a coordinate-path construction yields a joint density whose full conditionals satisfy explicit quadratic Kullback--Leibler bounds. We extend the analysis to stationary discounted problems and derive martingale and policy-improvement characterizations for learning the frozen response maps. A two-player linear-quadratic example illustrates the compatibility criterion and the associated learning procedure.
\end{abstract}



\section{Introduction}
\label{sec:intro}

A common continuous-time formulation of reinforcement learning for diffusion models replaces point controls by randomized policies and regularizes the objective by entropy.  The exploratory HJB formulation was developed in \cite{WZZ20}, and martingale methods for policy evaluation, policy gradient, and $q$-learning were subsequently established in \cite{JZ22a,JZ22b,JZ23}; the sampling formulation used in the $q$-learning martingale results was corrected in \cite{JZ23err}.  Quantitative regret and adaptive-exploration questions have since been studied in \cite{TangZhou24,HuangJZ24,HuangZhou25}.  A complementary random-measure formulation gives a continuous-time limit for discretely sampled exploratory actions \cite{BenderThuan26}, while continuous-time learning with regret guarantees has also been developed for episodic continuous-time Markov decision processes \cite{GaoZhou26}.

Regularization also has a substantial game-theoretic literature.  Reinforcement and regularization generate smooth learning dynamics in finite games \cite{MertikopoulosSandholm16}, and entropy regularization has been used to stabilize learning in mean field games \cite{GuoXuZariphopoulou22}.  In stochastic games, entropy-regularized zero-sum formulations have been studied in discrete time \cite{SavasAhmadiTanakaTopcu19}.  Continuous-time exploratory games are understood mainly in structured settings: entropy-regularized LQG mean field games yield exploratory controls and finite-population approximate Nash equilibria \cite{FirooziJaimungal22}, while linear-quadratic two-person zero-sum stochastic differential games admit Gaussian saddle policies and reinforcement-learning algorithms \cite{SunJia23}.  In all of these equilibrium formulations, a player's optimization problem is defined against the other players' continuation behavior.

That continuation dependence is also present in a finite $N$-player stochastic differential game.  Under the standard Nash definition, player $i$ optimizes against the complete continuation policies of the other players.  Entropy regularization changes the pointwise form of the optimizer, but it does not remove this dependence: although the opponents' current mixed action enters the Hamiltonian locally, the derivatives of the value function depend on their future policies.  A response rule parametrized only by the opponents' currently realized actions therefore defines a different strategic object.

We study that object directly.  At a state-time pair $(t,x)$, player $i$ observes the contemporaneous opponent profile $u^{-i}$ and solves an entropy-regularized control problem in which those opponent actions are held fixed over the continuation horizon.  This single-agent problem determines a Gibbs response $\pi^{i,v}_*(\cdot\mid u^{-i};t,x)$.  The family of responses is thus specified through full conditional densities rather than through a product of independent policies.  We call a joint law having these full conditionals a \emph{time-local conditional equilibrium}.  The terminology is meant to distinguish the model from a standard Nash equilibrium: the player's own future state and rewards are evaluated dynamically, while the opponents' future actions are not modeled beyond their current realized profile.

The resulting existence question is a compatibility problem for conditionally specified distributions.  Full conditional densities need not be the conditionals of any common joint law, a point emphasized in the classical compatibility literature \cite{ArnoldPress89,ArnoldCastilloSarabia92} and in later constructions for conditionally specified models \cite{Chen10}.  For the Gibbs responses above, we show that compatibility is equivalent to the entropy-scaled action gradients of the playerwise Hamiltonians forming a conservative field.  On a simply connected action domain, this reduces to a cross-partial condition and yields the joint density by a Poincar\'e construction.  The same condition is an entropy-weighted potential condition, connecting the model with exact potential games \cite{MondererShapley96} and with Gibbs models of strategic interaction \cite{Blume93,BlumeDurlauf03}.

A compatible joint law may be correlated, but the construction is not a correlated equilibrium in the usual incentive sense.  Correlated equilibrium conditions on a player's own recommendation and the induced posterior over the other players' actions \cite{Aumann74}; entropy and bounded-rationality variants modify that incentive structure in different ways \cite{OrtizSchapireKakade07,CernyAnZhang22,LiYuDorflerLygeros24}.  Here the direction of conditioning is reversed: each player's action law is specified conditional on the contemporaneous actions of all other players, and compatibility asks whether these full conditionals can be assembled into one joint distribution.

When exact compatibility fails, we quantify the defect rather than impose an equilibrium by construction.  A coordinate-path potential gives an approximate joint realization whose conditional Kullback--Leibler error is quadratic in a uniform cross-partial gap.  We also obtain the corresponding value loss for the frozen one-player problems and stationary discounted analogues.  Finally, the frozen response maps admit $q$-function, policy-improvement, and martingale characterizations that can be estimated from grid-sampled exploratory data.  These learning results identify the conditional response maps of the model; they do not solve the fixed point in continuation policies required by the standard Nash problem.

The paper is organized as follows.  Section~\ref{sec:prelim} fixes notation and regularity conventions.  Section~\ref{sec:problem} introduces the classical and exploratory games, records the standard Nash benchmark, and defines the frozen continuation assessment and the time-local conditional equilibrium.  Section~\ref{sec:equilibrium} develops the compatibility criterion, its potential-game interpretation, special cases, large-exploration estimates, and the approximate joint construction.  Section~\ref{sec:ergodic} gives the stationary discounted counterpart.  Sections~\ref{sec:qlearning} and \ref{sec:algorithms} develop the $q$-function characterizations and the associated learning procedures.  Section~\ref{sec:numerical} studies a two-player linear-quadratic benchmark, and Section~\ref{sec:concluding} concludes.  Appendix~\ref{app:DPP} records the weak dynamic programming argument used in the finite-horizon formulation.

\section{Preliminaries}\label{sec:prelim}
Let $(\Omega,\mathscr{F},\{\mathscr{F}_t\}_{t\in[0,T]},\mathbb{P})$ be a filtered probability space satisfying the usual conditions. Throughout the paper, $T>0$ is a fixed finite time horizon and $N\in\mathbb{N}$ is the number of players. The joint state space is $\mathbb{R}^N$, with $x=(x^1,\ldots,x^N)^\intercal\in\mathbb{R}^N$ denoting the joint state, while each player has the action space $U\subset\mathbb{R}^m$.

We use $\partial_t$ for the time derivative and $\partial_x$ or $\nabla_x$ for the gradient in $x$; similarly, $\partial_x^2$ and $D_x^2$ both denote the Hessian. For matrices $A$ and $B$ of the same dimension, $A:B=\operatorname{tr}(AB^\intercal)$ is their Frobenius inner product.

A superscript $i\in\{1,\ldots,N\}$ indexes the players. Thus $x^i$ and $u^i$ are the state and action of player $i$, while $\gamma^i>0$ and $\beta^i\ge0$ are its exploration weight and discount rate. For $u=(u^1,\ldots,u^N)\in U^N$, the vector $u^{-i}\in U^{N-1}$ contains all action coordinates except $u^i$. Tildes, as in $\tilde V^i$ and $\tilde J^i$, distinguish exploratory quantities from their classical counterparts, and the subscript $*$ denotes optimal objects.

We use the standard parabolic H\"older spaces; see \cite{Krylov87,GT01}. For $\alpha\in(0,1)$, $C^{1,2}([0,T]\times\mathbb R^N)$ consists of functions that are continuously differentiable once in $t$ and twice in $x$, while $C^{1+\alpha/2,2+\alpha}([0,T]\times\mathbb R^N)$ denotes the corresponding parabolic H\"older space. For functions of $x$ alone, $C^{2+\alpha}(\mathbb R^N)$ has the usual meaning. More generally, for an integer $k\ge0$, $C^{k+\alpha}$ denotes $C^k$ functions whose derivatives of order $k$ are $\alpha$-H\"older continuous.

For coefficients depending also on the actions, we write $C^{1,2+\alpha,\ell}([0,T]\times\mathbb R^N\times U^N)$, $\ell\in\{1,2\}$, for functions that are once continuously differentiable in $t$, have $\alpha$-H\"older continuous second derivatives in $x$, and are $\ell$ times continuously differentiable in $u$. Likewise, $C^{\alpha/2,\alpha,1+\alpha}([0,T]\times\mathbb R^N\times U)$ denotes the analogous mixed H\"older class with one action derivative. A subscript $\mathrm{loc}$ means that the stated regularity holds on compact subsets.

\section{Problem Formulation}
\label{sec:problem}

\subsection{Classical \texorpdfstring{$N$}{N}-Player Game}

For $N\in\mathbb N$ and $T>0$, let $X_t=(X_t^1,\ldots,X_t^N)^\intercal\in\mathbb R^N$ be the joint state process. The state is observed by all players and satisfies
\begin{align}\label{eq sec3 state}
dX_t^i = b^i(t,X_t,\boldsymbol{\alpha}_t)dt + \sigma^i(t,X_t,\boldsymbol{\alpha}_t)dW_t^i, \quad i=1,\ldots,N,
\end{align}
where $\boldsymbol{\alpha}_\cdot:=(\alpha^1_\cdot,\ldots,\alpha^N_\cdot)^\intercal$ is the $\mathbb{F}$-adapted, $U$-valued control vector, and $W^1,\ldots,W^N$ are mutually independent one-dimensional standard Brownian motions. The coefficients $b^i:[0,T]\times\mathbb{R}^N\times U^N\to\mathbb{R}$ and $\sigma^i:[0,T]\times\mathbb{R}^N\times U^N\to\mathbb{R}$ are measurable in all variables.

Equivalently, for $s\in[t,T]$,
\begin{align*}
X_s^{t,x,\boldsymbol{\alpha},i} = x^i + \int_t^s b^i(r,X_r^{t,x,\boldsymbol{\alpha}},\boldsymbol{\alpha}_r)dr + \int_t^s\sigma^i(r,X_r^{t,x,\boldsymbol{\alpha}},\boldsymbol{\alpha}_r)dW_r^i.
\end{align*}

We restrict the running reward $f^i$ to depend on player $i$'s own action, although the state dynamics may depend on the full control profile. Allowing $f^i$ to depend on $\boldsymbol\alpha$ would add further coupling to the HJB system and is not considered here.

For a fixed continuation control profile $\alpha^{-i}$ of the other players, player $i$ evaluates
\begin{align*}
J^i(t,x;\alpha^i,\alpha^{-i})
:=\mathbb E\bigg[\int_t^T e^{-\beta^i(s-t)}f^i(s,X_s^{t,x,\boldsymbol\alpha},\alpha_s^i)ds
+e^{-\beta^i(T-t)}g^i(X_T^{t,x,\boldsymbol\alpha})\bigg],
\end{align*}
and $V^i(t,x;\alpha^{-i})=\sup_{\alpha^i}J^i(t,x;\alpha^i,\alpha^{-i})$.  A classical Nash equilibrium is a profile $\boldsymbol\alpha^*$ for which $\alpha^{i,*}$ attains this best response against $\alpha^{-i,*}$ for every $i$.  The exploratory formulation below replaces each deterministic control by a randomized Markov policy; the corresponding best-response HJB is recorded explicitly in Section~\ref{sec:problem}.

We impose the following standing conditions on the state coefficients and rewards.

\begin{assumption}\label{assumption coe}
\begin{enumerate}
    \item $b^i,\sigma^i,f^i,g^i$ are continuous in their respective arguments.
    \item There exists a constant $L>0$ such that for all $(t,u)\in[0,T]\times U^N$ and $x,x'\in\mathbb{R}^N$,         
    \begin{align*}
        \sum_{i=1}^N\big[|b^i(t,x,u)-b^i(t,x',u)|+|\sigma^i(t,x,u)-\sigma^i(t,x',u)|\big]\leq L|x-x'|.
    \end{align*}
    \item There exists a constant $L_u>0$ such that for all $(t,x)\in[0,T]\times\mathbb{R}^N$ and $u,v\in U^N$,
    \begin{align*}
        \sum_{i=1}^N\big[|b^i(t,x,u)-b^i(t,x,v)|+|\sigma^i(t,x,u)-\sigma^i(t,x,v)|\big]\leq L_u|u-v|.
    \end{align*}
    \item There exists a constant $C>0$ such that for all $(t,x,u)\in[0,T]\times\mathbb{R}^N\times U^N$, 
    \begin{align*}
        & \sum_{i=1}^N\big[|b^i(t,x,u)|+|\sigma^i(t,x,u)|\big]\leq C(1 + |x| + |u| ).
    \end{align*}
    \item There exist constants $C>0$, $p_f\ge2$, and $q_f\ge1$ such that,
    for all $(t,x,u^i)\in[0,T]\times\mathbb{R}^N\times U$ and
    $1\le i\le N$,
        \begin{align*}
            |f^i(t,x,u^i)|\leq C(1+|x|^{p_f}+|u^i|^{q_f}),\quad |g^i(x)|\leq C(1+|x|^{p_f}).
        \end{align*}
\end{enumerate}
\end{assumption}

For the classical HJB arguments below, we use the stronger regularity and ellipticity conditions that follow.

\begin{assumption}\label{assumption regularity}
In addition to Assumption~\ref{assumption coe}, the diffusion is uniformly elliptic: there exists $\sigma_0>0$ such that $\sum_{k=1}^N[\sigma^k(t,x,u)]^2\xi_k^2\ge\sigma_0^2|\xi|^2$ for all $(t,x,u,\xi)\in[0,T]\times\mathbb R^N\times U^N\times\mathbb R^N$. For some $\alpha\in(0,1)$,
\[
b^i,\sigma^i\in C^{1,2+\alpha,2}([0,T]\times\mathbb R^N\times U^N),\qquad
f^i\in C^{1,2+\alpha,1}([0,T]\times\mathbb R^N\times U),
\]
and $g^i\in C^{2+\alpha}(\mathbb R^N)$, with all indicated derivatives bounded.
\end{assumption}

\subsection{Exploratory Formulation}

Following \cite{WZZ20}, we recast the game in the framework of reinforcement learning, replacing each player's deterministic control with a stochastic policy to model the exploration--exploitation balance.

Denote by $\cP(U)$ the set of Borel probability measures on $U$ that are absolutely continuous with respect to Lebesgue measure, and identify each such measure with its density. A \emph{stochastic policy} $\pi$ maps $(t,x)\in[0,T]\times\mathbb{R}^N$ to a density $\pi(\cdot;t,x)\in\cP(U)$, jointly measurable in $(t,x,a)$, with $\mathrm{supp}(\pi)=U$.  For a standard profile $\boldsymbol\pi=(\pi^1,\ldots,\pi^N)$, the players randomize independently and we write
\[
\boldsymbol\pi(du;t,x):=\bigotimes_{j=1}^N\pi^j(du^j;t,x).
\]
A joint density used later for the time-local conditional equilibrium need not have this product form. The averaged state equation and the entropy term require mild regularity and integrability of the Markov densities, so we work with the following class.

\begin{definition}\label{def:admissible_policy}
We call a stochastic policy $\pi$ \emph{admissible}, and write $\pi\in\Pi$, when the following conditions hold.
\begin{enumerate}
\item $\pi(u;t,x)$ is continuous in $(t,x)$: as $(t',x')\to(t,x)$,
$\int_U|\pi(u;t,x)-\pi(u;t',x')|du\to 0$;
and Lipschitz continuous in $x$: there exists $C>0$ such that
$\int_U|\pi(u;t,x)-\pi(u;t,x')|du\leq C|x-x'|$ for all $x,x'\in\mathbb{R}^N$.
\item For all $(t,x)$: $\int_U|\log\pi(u;t,x)|\pi(u;t,x)du\leq C(1+|x|^p)$ for some $p\geq 2$, and $\int_U|u|^p\pi(u;t,x)du\leq C_p(1+|x|^p)$ for any $p\geq 1$.
\item Let $G^\pi$ be the \emph{action function} of $\pi$ (i.e., $a=G^\pi(t,x,Z_s)\sim\pi(\cdot;t,x)$ for $Z_s\sim \mathrm{Uniform}([0,1]^m)$); for any $p\geq 2$, $\int_{[0,1]^m}|G^\pi(t,x,z)-G^\pi(t,x',z)|^pdz\leq L_p|x-x'|^p$.
\end{enumerate}
\end{definition}

Because the entropy term is written relative to Lebesgue measure, exploratory policies are restricted to absolutely continuous laws. Whenever the relevant partition function is finite, the Gibbs optimizer derived below is again a density. The zero-exploration limit may leave this class and is not used in the compatibility analysis.

The exploratory diffusion is defined through the averaged coefficients, following \cite{WZZ20}.  This averaged formulation does not require a continuum of independently sampled actions.  When sampled action processes are used later for $q$-learning, we use the time-grid construction of the erratum \cite{JZ23err}.

The \emph{exploratory (average) state} $X^{t,x,\boldsymbol\pi}$ is defined by
\begin{align}\label{eq exploratory state}
X_s^{t,x,\pi,i} = x^i + \int_t^s\tilde{b}^i(r,X_r^{t,x,\boldsymbol{\pi}},\boldsymbol{\pi}_r)dr + \int_t^s\tilde{\sigma}^i(r,X_r^{t,x,\boldsymbol{\pi}},\boldsymbol{\pi}_r)dW_r^i,
\end{align}
where $\boldsymbol{\pi}_\cdot=\{\pi^1_\cdot,\ldots,\pi^N_\cdot\}$, and the averaged coefficients are
\begin{align*}
\tilde{b}^i(t,x,\boldsymbol{\pi})
&:= \int_{U^N} b^i(t,x,u)\boldsymbol{\pi}(u;t,x)du,\\
\tilde{\sigma}^i(t,x,\boldsymbol{\pi})
&:= \sqrt{\int_{U^N}[\sigma^i(t,x,u)]^2\,\boldsymbol{\pi}(u;t,x)du}.
\end{align*}
The averaged diffusion~\eqref{eq exploratory state} is the exploratory state formulation of \cite{WZZ20}.  It is used directly below; sampled action processes enter only in the learning section.

The exploratory performance functional of player $i$ is
\begin{align}\label{eq exploratory value}
\tilde{J}^i(t,x;\boldsymbol{\pi})
:= \mathbb{E}\bigg[&\int_t^T e^{-\beta^i(s-t)}\int_U
\big[f^i(s,X_s^{t,x,\boldsymbol{\pi}},u^i)
- \gamma^i\log\pi^i(u^i;s,X_s^{t,x,\boldsymbol{\pi}})\big]\nonumber\\
&\hspace{20mm}\times\pi^i(u^i;s,X_s^{t,x,\boldsymbol{\pi}})du^i\,ds
+ e^{-\beta^i(T-t)}g^i(X_T^{t,x,\boldsymbol{\pi}})\bigg].
\end{align}
where $\gamma^i>0$ is the exploration weight for player $i$. The regularization $-\gamma^i\log\pi^i$ penalizes policies with low entropy. The exploratory value function of player $i$ is
\begin{align*}
\tilde{V}^i(t,x;\pi^{-i}) := \sup_{\pi^i\in\Pi}\tilde{J}^i(t,x;\pi^i,\pi^{-i}).
\end{align*}

We write
\begin{align*}
\tilde{f}^i(t,x,\pi^i) &:= \int_U f^i(t,x,u^i)\pi^i(u^i;t,x)du^i,\\
\mathcal{E}(t,x,\pi^i) &:= \int_U\log\pi^i(u^i;t,x)\pi^i(u^i;t,x)du^i,
\end{align*}
so that $\mathcal E$ is the negative Shannon entropy of the policy density.

Under Assumption~\ref{assumption coe}, the exploratory SDE~\eqref{eq exploratory state} is well-posed for each admissible Markov profile $\boldsymbol\pi$.  Let $Z=(Z^1,\ldots,Z^N)$ be independent uniform seeds and set
\[
G^{\boldsymbol\pi}(t,x,Z):=\bigl(G^{\pi^1}(t,x,Z^1),\ldots,G^{\pi^N}(t,x,Z^N)\bigr).
\]
Then
\[
\tilde b^i(t,x,\boldsymbol\pi)=\mathbb E_Z[b^i(t,x,G^{\boldsymbol\pi}(t,x,Z))],
\qquad
\tilde\sigma^i(t,x,\boldsymbol\pi)=\|\sigma^i(t,x,G^{\boldsymbol\pi}(t,x,Z))\|_{L^2(Z)}.
\]
For $x,x'\in\mathbb R^N$, Assumption~\ref{assumption coe}(2)--(3), Jensen's inequality, and Definition~\ref{def:admissible_policy}(3) give
\begin{align*}
|\tilde b^i(t,x,\boldsymbol\pi)-\tilde b^i(t,x',\boldsymbol\pi)|
&\le L|x-x'|+L_u\,\|G^{\boldsymbol\pi}(t,x,Z)-G^{\boldsymbol\pi}(t,x',Z)\|_{L^1(Z)}\\
&\le C_{\boldsymbol\pi}|x-x'|.
\end{align*}
For the diffusion coefficient, the reverse triangle inequality in $L^2(Z)$ yields
\begin{align*}
|\tilde\sigma^i(t,x,\boldsymbol\pi)-\tilde\sigma^i(t,x',\boldsymbol\pi)|
&\le\|\sigma^i(t,x,G^{\boldsymbol\pi}(t,x,Z))-\sigma^i(t,x',G^{\boldsymbol\pi}(t,x',Z))\|_{L^2(Z)}\\
&\le C_{\boldsymbol\pi}|x-x'|.
\end{align*}
The linear-growth bound follows in the same way from Assumption~\ref{assumption coe}(4) and the moment condition in Definition~\ref{def:admissible_policy}(2).  Standard SDE estimates therefore give a unique strong solution and $\mathbb E[\sup_{s\in[t,T]}|X_s^{t,x,\boldsymbol\pi}|^2]<\infty$.

\subsection{Exploratory HJB Equation and the Standard Nash Benchmark}

Fix the continuation policies $\pi^{-i}$ of the players other than $i$. With these policies treated as fixed data, player $i$ faces a single-agent exploratory control problem. The weak dynamic programming principle takes the following form.

\begin{proposition}\label{prop:DPP}
Fix $\pi^{-i}\in\Pi^{N-1}$.  Suppose $\tilde V^i(\cdot,\cdot;\pi^{-i})$ is continuous and the admissible control class is taken with the standard stable closure under the concatenations required by the weak DPP.  For $(t,x)\in[0,T]\times\mathbb R^N$ and any $[t,T]$-valued stopping time $\tau$,
\begin{align}\label{eq:DPP_stopping}
\tilde{V}^i(t,x;\pi^{-i}) =& \sup_{\pi^i\in\Pi}\mathbb{E}\bigg[\int_t^\tau e^{-\beta^i(r-t)}\big(\tilde{f}^i(r,X_r^{t,x,\boldsymbol{\pi}},\pi_r^i) - \gamma^i\mathcal{E}(r,X_r^{t,x,\boldsymbol{\pi}},\pi_r^i)\big)dr\nonumber\\
&\qquad+ e^{-\beta^i(\tau-t)}\tilde{V}^i(\tau,X_\tau^{t,x,\boldsymbol{\pi}};\pi^{-i})\bigg].
\end{align}
In particular, for deterministic $s\in[t,T]$,
\begin{align}\label{eq:DPP_deterministic}
\tilde{V}^i(t,x;\pi^{-i}) =& \sup_{\pi^i\in\Pi}\mathbb{E}\bigg[\int_t^s e^{-\beta^i(r-t)}\big(\tilde{f}^i(r,X_r^{t,x,\boldsymbol{\pi}},\pi_r^i) - \gamma^i\mathcal{E}(r,X_r^{t,x,\boldsymbol{\pi}},\pi_r^i)\big)dr \nonumber\\
&\qquad+ e^{-\beta^i(s-t)}\tilde{V}^i(s,X_s^{t,x,\boldsymbol{\pi}};\pi^{-i})\bigg].
\end{align}
\end{proposition}

\begin{proof}
The claim is the continuous-value specialization of the weak dynamic programming principle of \cite[Thm.~3.1]{BouchardTouzi11}; see also \cite[Ch.~IV]{FlemSoner}.  The controlled exploratory diffusion is well posed as shown above, the admissible class is stable under deterministic concatenation and the stopping-time pastings used in the weak DPP, and the entropy running term is additive and integrable by Definition~\ref{def:admissible_policy}.  The upper and lower semicontinuous envelopes in the weak DPP therefore coincide with $\tilde V^i$.  Appendix~\ref{app:DPP} records the reduction to~\eqref{eq:DPP_stopping}.
\end{proof}

For $p\in\mathbb R^N$ and a symmetric matrix $A$, define
\begin{align}\label{eq Hamiltonian}
H^i(t,x,u,p,A) := \sum_{k=1}^N \big[b^k(t,x,u)p^k + \tfrac12\sigma^k(t,x,u)^2 A^{k,k}\big]+f^i(t,x,u^i),
\end{align}
and, for a fixed opponent profile,
\begin{align}\label{eq Hamiltonian tilde}
\tilde H^i(t,x,u^i,\pi^{-i},p,A)
:=\int_{U^{N-1}}H^i(t,x,u^i,u^{-i},p,A)\,\pi^{-i}(du^{-i};t,x).
\end{align}
Here $\pi^{-i}(du^{-i};t,x)=\bigotimes_{j\ne i}\pi^j(du^j;t,x)$.  Assuming a classical solution, the HJB equation is
\begin{align}\label{eq HJB 1}
0=&\;\partial_t\tilde V^i(t,x;\pi^{-i})-\beta^i\tilde V^i(t,x;\pi^{-i})\nonumber\\
&+\sup_{\mu^i\in\mathcal P(U)}\int_U\Big[\tilde H^i\big(t,x,u^i,\pi^{-i},\partial_x\tilde V^i,\partial_x^2\tilde V^i\big)-\gamma^i\log\mu^i(u^i)\Big]\mu^i(u^i)du^i.
\end{align}
Whenever the partition function is finite, the entropy variational formula gives the unique current best response
\begin{align}\label{eq:standard_gibbs_br}
\mathrm{BR}_i[\pi^{-i}](u^i;t,x)
=\frac{\exp\{\tilde H^i(t,x,u^i,\pi^{-i},\partial_x\tilde V^i,\partial_x^2\tilde V^i)/\gamma^i\}}
{\int_U\exp\{\tilde H^i(t,x,v^i,\pi^{-i},\partial_x\tilde V^i,\partial_x^2\tilde V^i)/\gamma^i\}\,dv^i}.
\end{align}

For comparison with the time-local model introduced below, we retain the usual exploratory Nash benchmark.

\begin{definition}\label{def:standard_nash}
A profile $\boldsymbol\pi^*=(\pi^{1,*},\ldots,\pi^{N,*})\in\Pi^N$ is a standard exploratory Nash equilibrium if, for every $i$,
\[
\tilde J^i(t,x;\pi^{i,*},\pi^{-i,*})\geq \tilde J^i(t,x;\pi^i,\pi^{-i,*})
\quad\text{for all }\pi^i\in\Pi
\]
and every initial condition $(t,x)$.
\end{definition}

Equation~\eqref{eq HJB 1} does not imply a reduction to the opponents' current actions.  Although $\pi^{-i}$ appears explicitly in the Hamiltonian only through $\pi^{-i}(\cdot;t,x)$, the derivatives of $\tilde V^i(t,x;\pi^{-i})$ generally depend on the continuation policies on $[t,T]$.  Thus the standard Nash problem is a fixed point in complete policies.

\subsection{Continuation Assessment and the Time-Local Response}\label{subsec:assessment}

For the compatibility results in Sections~\ref{sec:equilibrium}--\ref{sec:ergodic}, we take scalar actions, $U\subset\mathbb R$.  The vector-action version replaces scalar mixed partials by the corresponding block-Jacobian symmetry conditions.

The model of strategic interaction studied in the remainder of the paper is the following.  Each player responds to the actions actually realized by the others at the current time, choosing a randomized action that is an entropy-regularized best response to the realized opponent profile.  No player can forecast how the others will act later, so the future is assessed with the opponents held at their realized actions.  This is the sense in which the model is short-sighted.  The player's own future rewards and the evolution of the state enter the assessment in full; only the opponents' future actions are not predicted.  We adopt this response rule as the model itself, not as an approximation to the Nash benchmark recorded above; the latter serves only to make precise what is, and what is not, being solved.  Two features organize the subsequent analysis.  Because the response conditions on realized actions rather than integrating against the opponents' continuation policies, the resulting family of conditional laws need not be consistent with any single joint law, and deciding when it is becomes the compatibility problem of Section~\ref{sec:equilibrium}.  The same conditional response is also the object that the learning sections identify.

The assessment is made precise as follows.  Fix $u^{-i}\in U^{N-1}$ and hold the opponents' actions at that value over the remaining horizon.  Player $i$ then faces an entropy-regularized single-agent control problem, with performance $\tilde J^i(t,x;\pi^i,u^{-i})$ and value $\tilde V^i(t,x;u^{-i})$.  The value solves
\begin{align}\label{eq HJB-explore}
0=&\;\partial_t\tilde V^i(t,x;u^{-i})-\beta^i\tilde V^i(t,x;u^{-i})\nonumber\\
&+\sup_{\pi^i\in\mathcal P(U)}\int_U\Big[H^i(t,x,u^i,u^{-i},\partial_x\tilde V^i,\partial_x^2\tilde V^i)-\gamma^i\log\pi^i(u^i;t,x)\Big]\pi^i(u^i;t,x)du^i,
\end{align}
or, after carrying out the inner maximization,
\begin{align}\label{eq explore HJB 2}
0=\partial_t\tilde V^i-\beta^i\tilde V^i+\gamma^i\log\int_U\exp\Big\{\frac1{\gamma^i}H^i(t,x,u^i,u^{-i},\partial_x\tilde V^i,\partial_x^2\tilde V^i)\Big\}du^i.
\end{align}
We write $v^i(t,x;u^{-i}):=\tilde V^i(t,x;u^{-i})$ and call $v=(v^1,\ldots,v^N)$ the continuation assessment.  It is a single-agent object, determined by the coefficients and by the value at which the opponents are held, with no reference to any equilibrium condition.  The assessment can therefore be computed before, and independently of, the compatibility question below.

Whenever classical regularity of~\eqref{eq HJB-explore} is used, it is assumed explicitly. Under compact action sets and the usual uniform parabolicity and H\"older conditions, standard parabolic theory gives $\tilde V^i(\cdot,\cdot;u^{-i})\in C^{1,2}([0,T]\times\mathbb R^N)$ for each fixed $u^{-i}$. For noncompact $U$, the corresponding coercivity and integrability conditions are also required.

Fix $(t,x)$ and let the current action profile be $u=(u^1,\ldots,u^N)$.  The effect of future states enters through $\partial_xv^i$ and $\partial_x^2v^i$, evaluated at $(t,x;u^{-i})$, while the current action profile remains $u$.  Define
\begin{align}\label{eq:local_Hamiltonian}
H_v^i(t,x,u):=H^i\big(t,x,u,\partial_xv^i(t,x;u^{-i}),\partial_x^2v^i(t,x;u^{-i})\big).
\end{align}
For fixed $u^{-i}$, player $i$'s time-local response is the maximizer of
\[
\mu^i\longmapsto \int_U\big[H_v^i(t,x,u^i,u^{-i})-\gamma^i\log\mu^i(u^i)\big]\mu^i(u^i)du^i.
\]
The entropy variational formula gives
\begin{align}\label{eq optimal policy}
\pi^{i,v}_*(u^i;t,x,u^{-i})
:=\frac{\exp\{H_v^i(t,x,u^i,u^{-i})/\gamma^i\}}
{\int_U\exp\{H_v^i(t,x,z^i,u^{-i})/\gamma^i\}\,dz^i},
\end{align}
whenever the denominator is finite.  The assessment is indexed by the opponents' actions, because $v^i=\tilde V^i(\cdot;u^{-i})$ is evaluated at the profile on which the response is conditioned.  This dependence is what distinguishes the response from a Nash best response, and it is also what produces the indirect terms in the cross-partial decomposition of Section~\ref{sec:equilibrium}.

\subsection{Time-Local Conditional Equilibrium}

The responses in \eqref{eq optimal policy} are specified separately for the $N$ players. They define a contemporaneous action law only when they can be realized as the full conditional densities of one joint distribution. We use the following terminology for such a realization.

\begin{definition}\label{def:tlce}
Fix $(t,x)$ and the continuation assessment $v$.  A positive joint density $\psi^v(\cdot;t,x)$ on $U^N$ is a \emph{time-local conditional equilibrium} if, for every $i$, its full conditional density satisfies
\[
\psi^v(u^i\mid u^{-i};t,x)=\pi^{i,v}_*(u^i;t,x,u^{-i})
\quad\text{for a.e. }u^{-i}.
\]
\end{definition}

The response is defined relative to the actions currently realized by the others, not relative to an opponent continuation policy. Thus the compatibility question at a fixed $(t,x)$ is whether the $N$ conditional specifications above determine one joint action law. There are three equivalent ways to express this requirement.

\begin{theorem}\label{thm:compatibility}
Let $\pi^{i,v}_*(u^i\mid u^{-i})$ be positive conditional densities, twice continuously differentiable on the interior of $U^N$, and assume the interior of $U^N$ is connected and simply connected.  The following are equivalent.
\begin{enumerate}
\item A time-local conditional equilibrium exists; equivalently, there is a positive integrable function $\psi$ such that
\begin{align}\label{eq:compat_joint}
\pi^{i,v}_*(u^i\mid u^{-i})=\frac{\psi(u)}{\int_U\psi(z^i,u^{-i})\,dz^i},\qquad i=1,\ldots,N.
\end{align}
\item There are positive functions $\phi^i:U^{N-1}\to(0,\infty)$ such that the common function
\[
\psi(u):=\pi^{i,v}_*(u^i\mid u^{-i})\phi^i(u^{-i})
\]
is independent of $i$ and satisfies $0<\int_{U^N}\psi(u)du<\infty$.
\item The cross-partial identities
\begin{align}\label{eq:conditional_cross}
\frac{\partial^2}{\partial u^j\partial u^i}\log\pi^{i,v}_*(u^i\mid u^{-i})
=\frac{\partial^2}{\partial u^i\partial u^j}\log\pi^{j,v}_*(u^j\mid u^{-j}),\qquad i\ne j,
\end{align}
hold, and a potential $\Phi$ of the closed one-form
\[
\omega:=\sum_{i=1}^N\partial_{u^i}\log\pi^{i,v}_*(u^i\mid u^{-i})\,du^i
\]
is normalizable: $0<\int_{U^N}e^{\Phi(u)}du<\infty$.
\end{enumerate}
When these conditions hold, the joint density is unique.
\end{theorem}

\begin{proof}
$(1)\Rightarrow(2)$.  Put $\phi^i(u^{-i})=\int_U\psi(z^i,u^{-i})dz^i$.  Then \eqref{eq:compat_joint} gives $\psi=\pi^{i,v}_*\phi^i$ for every $i$, and integrability is inherited from the joint density.

$(2)\Rightarrow(1)$.  Let $\psi$ denote the common integrable function in (2) and normalize it by $Z=\int_{U^N}\psi(u)du$.  Since $\int_U\pi^{i,v}_*(u^i\mid u^{-i})du^i=1$,
\[
\int_U\psi(u)du^i=\phi^i(u^{-i}).
\]
The normalization by $Z$ cancels in the conditional ratio, so the density $Z^{-1}\psi$ has the prescribed full conditionals.

$(1)\Rightarrow(3)$.  From
\[
\log\psi(u)=\log\pi^{i,v}_*(u^i\mid u^{-i})+\log\phi^i(u^{-i})
\]
we obtain $\partial_{u^i}\log\psi=\partial_{u^i}\log\pi^{i,v}_*$.  Differentiating with respect to $u^j$ and interchanging mixed derivatives of $\log\psi$ gives \eqref{eq:conditional_cross}.  Taking $\Phi=\log\psi$ gives the normalizability condition.

$(3)\Rightarrow(1)$.  The identities \eqref{eq:conditional_cross} say that $\omega$ is closed.  By the Poincar\'e lemma, there is a $C^2$ function $\Phi$ such that $d\Phi=\omega$.  Let
\[
\widehat\psi(u)=Z^{-1}e^{\Phi(u)},\qquad Z=\int_{U^N}e^{\Phi(v)}dv<\infty.
\]
For fixed $u^{-i}$,
\[
\partial_{u^i}\log\widehat\psi(u)=\partial_{u^i}\Phi(u)=\partial_{u^i}\log\pi^{i,v}_*(u^i\mid u^{-i}).
\]
Hence $\log\widehat\psi(u)-\log\pi^{i,v}_*(u^i\mid u^{-i})$ is independent of $u^i$.  Thus
\[
\widehat\psi(u)=c_i(u^{-i})\pi^{i,v}_*(u^i\mid u^{-i}).
\]
Integrating with respect to $u^i$ shows $c_i(u^{-i})=\int_U\widehat\psi(z^i,u^{-i})dz^i$, which is exactly \eqref{eq:compat_joint}.

Finally, if two positive joint densities have the same full conditionals, the derivative of the logarithm of their ratio with respect to every coordinate $u^i$ is zero.  Connectedness implies that the ratio is constant, and normalization makes it one.
\end{proof}

On compact $U^N$, the normalizability condition in part~(3) is automatic for continuous $\Phi$. On a non-simply-connected domain, the cross-partial identities give only local exactness; global compatibility additionally requires the periods of $\omega$ along closed loops to vanish.

\section{Compatibility Analysis}\label{sec:equilibrium}

At a fixed $(t,x)$, the time-local responses are the Gibbs densities~\eqref{eq optimal policy}.  The compatibility theorem above turns their joint realizability into a condition on the current-action Hamiltonians.

\subsection{Hamiltonian Cross-Partial Criterion}

For the Gibbs response \eqref{eq optimal policy}, differentiation with respect to the player's own action removes the normalizing factor. Consequently, the compatibility identities in Theorem~\ref{thm:compatibility} reduce to the entropy-scaled Hamiltonian condition
\begin{align}\label{eq:cross_partial_q}
\frac1{\gamma^i}\frac{\partial^2 H_v^i}{\partial u^j\partial u^i}(t,x,u)
=\frac1{\gamma^j}\frac{\partial^2 H_v^j}{\partial u^i\partial u^j}(t,x,u),\qquad i\ne j.
\end{align}
This condition is necessary and sufficient, apart from the normalizability of the resulting potential.

\begin{theorem}\label{thm:compat_q}
Assume the hypotheses of Theorem~\ref{thm:compatibility}. The responses in~\eqref{eq optimal policy} admit a joint realization exactly when~\eqref{eq:cross_partial_q} holds and the resulting potential is normalizable. If $U$ is compact, normalizability is automatic.
\end{theorem}

\begin{proof}
For fixed $(t,x,u^{-i})$,
\[
\log\pi^{i,v}_*(u^i\mid u^{-i})
=\frac1{\gamma^i}H_v^i(t,x,u)-\log\int_U\exp\{H_v^i(t,x,z^i,u^{-i})/\gamma^i\}dz^i.
\]
The logarithm of the normalizing factor depends on $u^{-i}$ but not on $u^i$.  Therefore
\[
\partial_{u^i}\log\pi^{i,v}_*=\frac1{\gamma^i}\partial_{u^i}H_v^i,
\]
and hence, for $j\ne i$,
\[
\partial_{u^j}\partial_{u^i}\log\pi^{i,v}_*
=\frac1{\gamma^i}\partial_{u^j}\partial_{u^i}H_v^i.
\]
Substitution in Theorem~\ref{thm:compatibility} gives~\eqref{eq:cross_partial_q}.  If the condition holds, the closed one-form
\[
\omega_v:=\sum_{i=1}^N\frac1{\gamma^i}\partial_{u^i}H_v^i(t,x,u)\,du^i
\]
has a potential $\Phi_v$.  The compatible density is $Z_v^{-1}e^{\Phi_v}$ whenever $Z_v=\int_{U^N}e^{\Phi_v(u)}du<\infty$.
\end{proof}

On a simply connected action domain, the same criterion can be written as a potential condition.

\begin{proposition}\label{prop:weighted_potential}
Condition~\eqref{eq:cross_partial_q} holds if and only if, on a simply connected action domain, there is a scalar function $\Phi_v$ such that
\begin{align}\label{eq:weighted_potential}
\partial_{u^i}\Phi_v(t,x,u)=\frac1{\gamma^i}\partial_{u^i}H_v^i(t,x,u),\qquad i=1,\ldots,N.
\end{align}
If $\gamma^1=\cdots=\gamma^N=\gamma$ and the time-slice Hamiltonian game is an exact potential game with potential $P_v$, then $\Phi_v=P_v/\gamma$ and the compatible joint law is the Gibbs density $\psi^v\propto e^{P_v/\gamma}$.
\end{proposition}

\begin{proof}
Equation~\eqref{eq:cross_partial_q} is exactly the closedness condition for the one-form $\omega_v$ in the proof of Theorem~\ref{thm:compat_q}.  Poincar\'e's lemma gives~\eqref{eq:weighted_potential}.  Conversely, mixed derivatives of $\Phi_v$ commute, which gives~\eqref{eq:cross_partial_q}.  In the equal-temperature potential case, $\partial_{u^i}(P_v/\gamma)=\gamma^{-1}\partial_{u^i}H_v^i$, so the last claim follows from the construction of the joint density.
\end{proof}

The potential formulation also clarifies how the compatibility problem relates to nearby game-theoretic notions.

\begin{remark}\label{rmk:potential_gibbs}
The probabilistic part of the construction is the compatibility problem for full conditional distributions \cite{ArnoldPress89,ArnoldCastilloSarabia92}.  Its game-theoretic interpretation is close to Gibbs-equilibrium models based on conditional choice laws \cite{Blume93,BlumeDurlauf03}.  Proposition~\ref{prop:weighted_potential} identifies the additional structure imposed here by the HJB Hamiltonians and connects it with potential games \cite{MondererShapley96}.  Although a compatible law may be correlated, the condition is not the correlated-equilibrium incentive condition: the direction of conditioning is different, as noted in the introduction.
\end{remark}

\subsection{Structure of the Cross-Partial}

The mixed derivative in~\eqref{eq:cross_partial_q} differentiates $H_v^i$ in the own action $u^i$ and an opponent action $u^j$.  Since the assessment $v^i=\tilde V^i(\cdot;u^{-i})$ is indexed by the opponents' actions, the derivative has two sources: the explicit dependence of the coefficients on the action profile, and the dependence carried by the continuation value.  The decomposition below separates them.  Its indirect part is present precisely because the assessment is computed with the opponents held at their realized actions.

\begin{proposition}\label{prop:decomposition}
For $j\ne i$, assume the classical regularity used in Section~\ref{subsec:assessment} and that $\tilde V^i$ is $C^1$ in $u^{-i}$.  For the optimal $q$-function
\[
q^i_*=\partial_t\tilde V^i+H^i(t,x,u,\partial_x\tilde V^i,\partial_x^2\tilde V^i)-\beta^i\tilde V^i,
\]
we have
\begin{multline}\label{eq:cross_partial_decomp}
\frac{\partial^2 q^i_*}{\partial u^j\partial u^i}
=\underbrace{\sum_{k=1}^N\left[
\frac{\partial^2b^k}{\partial u^j\partial u^i}(\nabla_x\tilde V^i)^k
+\left(\frac{\partial\sigma^k}{\partial u^i}\frac{\partial\sigma^k}{\partial u^j}
+\sigma^k\frac{\partial^2\sigma^k}{\partial u^j\partial u^i}\right)(D_x^2\tilde V^i)^{k,k}
\right]}_{\text{direct effect}}\\
+\underbrace{\sum_{k=1}^N\left[
\frac{\partial b^k}{\partial u^i}\frac{\partial(\nabla_x\tilde V^i)^k}{\partial u^j}
+\sigma^k\frac{\partial\sigma^k}{\partial u^i}\frac{\partial(D_x^2\tilde V^i)^{k,k}}{\partial u^j}
\right]}_{\text{indirect effect}}.
\end{multline}
\end{proposition}

\begin{proof}
Differentiate the definition of $q^i_*$ first with respect to $u^i$.  The terms $\partial_t\tilde V^i$ and $-\beta^i\tilde V^i$ do not depend on $u^i$, because $u^{-i}$ is the frozen parameter.  Hence
\[
\frac{\partial q^i_*}{\partial u^i}
=\sum_{k=1}^N\left[
\frac{\partial b^k}{\partial u^i}(\nabla_x\tilde V^i)^k
+\sigma^k\frac{\partial\sigma^k}{\partial u^i}(D_x^2\tilde V^i)^{k,k}
\right]+\frac{\partial f^i}{\partial u^i}.
\]
Now differentiate with respect to $u^j$, $j\ne i$.  The term $\partial f^i/\partial u^i$ is independent of $u^j$.  Applying the product rule to the remaining terms gives the derivatives of the coefficients displayed in the first line of~\eqref{eq:cross_partial_decomp} and the derivatives of the frozen-value jets displayed in the second line.  This is exactly~\eqref{eq:cross_partial_decomp}.
\end{proof}

\subsection{Special Cases}

The criterion becomes immediate when the dynamics are decoupled across players.

\begin{theorem}\label{thm:decoupled}
Suppose $b^k(t,x,u)=b^k(t,x^k,u^k)$ and $\sigma^k(t,x,u)=\sigma^k(t,x^k,u^k)$ for every $k$, and $f^i(t,x,u^i)$ depends only on $u^i$ in the action variable.  Then the time-local conditional response of player $i$ is independent of $u^{-i}$, and a time-local conditional equilibrium exists.  Its joint density is the product of the individual Gibbs responses.
\end{theorem}

\begin{proof}
Under the stated structure the frozen HJB~\eqref{eq HJB-explore} for player $i$ contains neither $u^{-i}$ nor $x^{-i}$, so $\tilde V^i$ does not depend on $u^{-i}$ and $\partial_{u^j}\partial_x\tilde V^i=0$ for $j\ne i$.  Hence $H_v^i(t,x,u)$ is a sum of terms, each depending on a single action coordinate.  All terms involving $u^j$, $j\ne i$, are independent of $u^i$ and therefore cancel from the Gibbs normalization in~\eqref{eq optimal policy}.  Hence $\pi^{i,v}_*(u^i\mid u^{-i})=\pi^{i,v}_*(u^i)$.  The product $\prod_i\pi^{i,v}_*(u^i)$ has these full conditionals and is therefore a time-local conditional equilibrium.
\end{proof}

Decoupling is sufficient for compatibility, but permutation symmetry alone is not.
\begin{remark}\label{rmk:symmetry}
For example, on $U=[-1,1]$ with equal exploration weights, the smooth symmetric pair
\[
H^1(u^1,u^2)=u^1(u^2)^2,\qquad H^2(u^1,u^2)=u^2(u^1)^2
\]
satisfies $H^2(u^1,u^2)=H^1(u^2,u^1)$, but
$\partial_{u^2}\partial_{u^1}H^1=2u^2$ and
$\partial_{u^1}\partial_{u^2}H^2=2u^1$.  Thus an additional potential condition is needed.  The example is a statement about the criterion itself; whether a given pair of Hamiltonians arises from an assessment of the present model is a separate question.
\end{remark}

For two players the potential can be written explicitly by integrating along the two coordinate directions.

\begin{theorem}\label{thm:N2}
Let $N=2$, $U\subset\mathbb R$ be compact, convex, and contain the origin.  If $H_v^1,H_v^2\in C^2(U^2)$, then a time-local conditional equilibrium exists if and only if
\begin{align}\label{eq:compat_N2}
\frac1{\gamma^1}\frac{\partial^2H_v^1}{\partial u^2\partial u^1}
=\frac1{\gamma^2}\frac{\partial^2H_v^2}{\partial u^1\partial u^2}.
\end{align}
When this holds, its joint density is proportional to
\begin{align*}
\exp\left\{\int_0^{u^1}\frac1{\gamma^1}\frac{\partial H_v^1}{\partial v^1}(v^1,0)dv^1
+\int_0^{u^2}\frac1{\gamma^2}\frac{\partial H_v^2}{\partial v^2}(u^1,v^2)dv^2\right\}.
\end{align*}
\end{theorem}

\begin{proof}
Necessity is Theorem~\ref{thm:compat_q}.  For sufficiency, set
\[
F^1(u^1)=\frac1{\gamma^1}\partial_{u^1}H_v^1(u^1,0),\qquad
F^2(u^1,u^2)=\frac1{\gamma^2}\partial_{u^2}H_v^2(u^1,u^2),
\]
and define
\[
\Phi(u^1,u^2)=\int_0^{u^1}F^1(v^1)dv^1+\int_0^{u^2}F^2(u^1,v^2)dv^2.
\]
Clearly $\partial_{u^2}\Phi=\frac{1}{\gamma^2}\partial_{u^2}H_v^2$.  For the other derivative,
\begin{align*}
\partial_{u^1}\Phi
&=\frac1{\gamma^1}\partial_{u^1}H_v^1(u^1,0)
+\int_0^{u^2}\frac1{\gamma^2}\partial_{u^1}\partial_{v^2}H_v^2(u^1,v^2)dv^2\\
&=\frac1{\gamma^1}\partial_{u^1}H_v^1(u^1,0)
+\int_0^{u^2}\frac1{\gamma^1}\partial_{v^2}\partial_{u^1}H_v^1(u^1,v^2)dv^2\\
&=\frac1{\gamma^1}\partial_{u^1}H_v^1(u^1,u^2),
\end{align*}
where the second line uses~\eqref{eq:compat_N2}.  Thus $d\Phi=\omega_v$, and compactness of $U^2$ makes $e^\Phi$ normalizable.  Theorem~\ref{thm:compatibility} gives the required joint density.
\end{proof}

\subsection{Large-\texorpdfstring{$\gamma$}{gamma} Analysis for Frozen Responses}

We next examine the compatibility defect when the exploration weights are large. The conclusion is local in the state variable because the required bounds on the frozen values are local.

\begin{proposition}\label{prop:large_gamma_approx}
Suppose $U\subset\mathbb R$ is compact and Assumptions~\ref{assumption coe}--\ref{assumption regularity} hold.  For each $R>0$, assume that the classical frozen values satisfy, uniformly for sufficiently large $\gamma^i$,
\begin{align}\label{eq:uniform_frozen_jets}
\sup_{[0,T]\times B_R\times U^{N-1}}\Big(
|\partial_x\tilde V^i|+|D_x^2\tilde V^i|
+\sum_{j\ne i}\big[|\partial_x\partial_{u^j}\tilde V^i|+|D_x^2\partial_{u^j}\tilde V^i|\big]
\Big)\le C_R.
\end{align}
Then, with $\gamma=\min_i\gamma^i$,
\begin{align*}
\sup_{[0,T]\times B_R\times U^N}\left|\pi_*^i(u^i\mid u^{-i};t,x)-\frac1{|U|}\right|&=O_R((\gamma^i)^{-1}),\\
\max_{i\ne j}\sup_{[0,T]\times B_R\times U^N}|\Delta^{ij}(t,x,u)|&=O_R(\gamma^{-1}).
\end{align*}
\end{proposition}

\begin{proof}
Fix $R$.  By compactness of $U$, Assumption~\ref{assumption coe}, and the first two bounds in~\eqref{eq:uniform_frozen_jets}, there is $M_R<\infty$, independent of sufficiently large $\gamma^i$, such that
\[
|H^i(t,x,u,\partial_x\tilde V^i,D_x^2\tilde V^i)|\le M_R
\quad\text{on }[0,T]\times B_R\times U^N.
\]
This bound implies
\[
|U|e^{-M_R/\gamma^i}\le
\int_Ue^{H^i(t,x,z^i,u^{-i},\partial_x\tilde V^i,D_x^2\tilde V^i)/\gamma^i}dz^i
\le |U|e^{M_R/\gamma^i}.
\]
The Gibbs formula therefore gives
\[
\frac{e^{-2M_R/\gamma^i}}{|U|}\le \pi_*^i(u^i\mid u^{-i};t,x)
\le\frac{e^{2M_R/\gamma^i}}{|U|},
\]
and the first estimate follows from $e^{2M_R/\gamma^i}-1=O_R((\gamma^i)^{-1})$.

For the second estimate, Proposition~\ref{prop:decomposition} expresses
$\partial_{u^j}\partial_{u^i}q_*^i$ as the sum of direct coefficient derivatives and terms containing
$\partial_x\partial_{u^j}\tilde V^i$ and $D_x^2\partial_{u^j}\tilde V^i$.
Assumption~\ref{assumption regularity} and~\eqref{eq:uniform_frozen_jets} thus give
\[
\sup_{[0,T]\times B_R\times U^N}
|\partial_{u^j}\partial_{u^i}q_*^i|\le C_R,
\qquad i\ne j,
\]
with $C_R$ independent of sufficiently large exploration weights.  Hence
\[
|\Delta^{ij}|
\le \frac{C_R}{\gamma^i}+\frac{C_R}{\gamma^j}
\le \frac{2C_R}{\gamma},
\]
which proves the second estimate.
\end{proof}

The estimate is local because the uniform jet bounds are imposed on $B_R$. It concerns only the frozen conditional specifications: large exploration makes these specifications close to uniform and approximately compatible, but does not imply convergence of the standard Nash problem.

\subsection{Approximate Joint Realization}\label{sec:approx_CE}

A small curl defect allows the frozen Gibbs specifications to be approximated by the full conditionals of one joint density. To measure this defect, for $i\ne j$ let
\begin{align}\label{eq:compat_gap}
\Delta^{ij}(t,x,u) := \frac{1}{\gamma^i}\frac{\partial^2 q^i_*}{\partial u^j \partial u^i}(t,x,u) - \frac{1}{\gamma^j}\frac{\partial^2 q^j_*}{\partial u^i \partial u^j}(t,x,u).
\end{align}

We call $\Delta^{ij}$ the \emph{compatibility gap}. The frozen Gibbs specifications are exactly compatible when $\Delta^{ij}\equiv0$ for every $i\ne j$ and the associated potential is normalizable.

A uniform bound on the compatibility gaps produces a compatible joint law whose full conditionals stay close to the frozen Gibbs responses.

\begin{theorem}\label{thm:approx_CE}
Suppose $U\subset\mathbb{R}$ is compact, convex, and contains the origin, and $q^i_*$ is twice continuously differentiable in $u$ for each $i$. Assume
\begin{align*}
\sup_{(t,x,u) \in [0,T] \times \mathbb{R}^N \times U^N} |\Delta^{ij}(t,x,u)| \le \varepsilon, \quad \forall i \neq j,
\end{align*}
then there exists a joint density $\hat{\psi}(\cdot;t,x)\in\cP(U^N)$ for each $(t,x)\in[0,T]\times\mathbb R^N$ such that
\begin{enumerate}
\item the conditional distributions $\hat{\pi}^i(\cdot\mid u^{-i};t,x)$ of $\hat\psi$ satisfy
\begin{align*}
D_{\mathrm{KL}}(\hat{\pi}^i(\cdot|u^{-i}; t,x) \| \pi^i_*(\cdot|u^{-i}; t,x)) \le (e^\delta - 1)^2, \quad \forall i, u^{-i}, (t,x) \in [0,T] \times \mathbb{R}^N,
\end{align*}
where $\delta:=2(N-1)\varepsilon|U|^2$. If $\delta\le1$ (equivalently, $\varepsilon\le[2(N-1)|U|^2]^{-1}$), then
\begin{align*}
D_{\mathrm{KL}}(\hat{\pi}^i(\cdot|u^{-i}; t,x) \| \pi^i_*(\cdot|u^{-i}; t,x)) \le 16(N-1)^2 \varepsilon^2 |U|^4, \quad \forall (t,x) \in [0,T] \times \mathbb{R}^N.
\end{align*}

\item for the frozen one-player response problem, define
\begin{align*}
\varepsilon^i_{\mathrm{loc}}
:=\sup_{(t,x,u^{-i})}\gamma^i D_{\mathrm{KL}}(\hat{\pi}^i(\cdot\mid u^{-i};t,x) \| \pi^i_*(\cdot\mid u^{-i};t,x)).
\end{align*}
Then $\varepsilon^i_{\mathrm{loc}}\le \gamma^i(e^\delta-1)^2$; if $\delta\le1$, this becomes $\varepsilon^i_{\mathrm{loc}}\le16\gamma^i(N-1)^2\varepsilon^2|U|^4$.
\end{enumerate}
When $\varepsilon=0$, $\hat\psi$ realizes all frozen Gibbs conditionals exactly.
\end{theorem}

\begin{proof}
Fix $(t,x)$. The piecewise-linear coordinate path from $0$ to $u$ remains in $U^N$ by convexity and the assumption $0\in U$. Along this path, set
\begin{align*}
\Phi(u) := \sum_{k=1}^N \int_0^{u^k} a^k(u^1, \ldots, u^{k-1}, v^k, 0, \ldots, 0) dv^k,
\end{align*}
where $a^k(u) := \frac{1}{\gamma^k}\frac{\partial q^k_*}{\partial u^k}(t,x,u)$. Set $\hat{\psi}(u; t,x) := e^{\Phi(u)}/Z(t,x)$ with $Z(t,x) = \int_{U^N} e^{\Phi(u)}du < \infty$ (since $U$ is compact and $\Phi$ is continuous).

Differentiating $\Phi$ and applying the fundamental theorem of calculus to compare $a^k(u)$ with $a^k(u^1,\ldots,u^k,0,\ldots,0)$ yields
\begin{align*}
\frac{\partial \Phi}{\partial u^k}(u) - a^k(u) = \sum_{j > k} \int_0^{u^j} \Delta^{jk}(u^1, \ldots, u^{j-1}, v^j, 0, \ldots, 0)\, dv^j.
\end{align*}
Since $|\Delta^{jk}| \le \varepsilon$, this gives the uniform bound
\begin{align}\label{eq:gradient_deviation}
\big|\tfrac{\partial \Phi}{\partial u^k}(u) - a^k(u)\big| \le (N-1)\varepsilon |U|.
\end{align}
For the log-ratio bound, the conditional density $\hat{\pi}^i(u^i|u^{-i}) = \hat{\psi}(u)/\int \hat{\psi}(v^i, u^{-i})dv^i$ satisfies $\frac{\partial}{\partial u^i}\log \hat{\pi}^i = \frac{\partial \Phi}{\partial u^i}$, while $\frac{\partial}{\partial u^i}\log \pi^i_* = a^i(u)$. By~\eqref{eq:gradient_deviation}, $\big|\frac{\partial}{\partial u^i}\log\frac{\pi^i_*}{\hat{\pi}^i}\big|\le (N-1)\varepsilon |U|$. Integrating from a reference point $u^i_0\in U$ gives $\big|\log \frac{\pi^i_*}{\hat{\pi}^i}(u^i|u^{-i}) - c(u^{-i})\big|\le (N-1)\varepsilon |U|^2 =: \eta$, where $c(u^{-i})=\log(\pi^i_*(u^i_0|u^{-i})/\hat{\pi}^i(u^i_0|u^{-i}))$. Since both $\pi^i_*$ and $\hat{\pi}^i$ integrate to 1 and $\pi^i_*/\hat{\pi}^i\in[e^{c-\eta},e^{c+\eta}]$, normalization forces $|c(u^{-i})|\le\eta$, giving $|\log(\pi^i_*/\hat{\pi}^i)| \le 2\eta =: \delta = 2(N-1)\varepsilon |U|^2$.

The KL bound via $\chi^2$-divergence follows: setting $h:=\log(\pi^i_*/\hat{\pi}^i)$ with $|h|\le\delta$, we have $D_{\mathrm{KL}}(\hat{\pi}^i \| \pi^i_*) \le \int_U \pi^i_*(e^{-h}-1)^2\,du^i \le (e^\delta - 1)^2$.

For the local entropy-regularized variational objective at $(t,x,u^{-i})$, the optimal value under $\pi^i_*$ is $0$ (by Gibbs normalization $\int e^{q^i_*/\gamma^i}du^i = 1$). The value under $\hat{\pi}^i$ is $\int[q^i_*-\gamma^i\log\hat{\pi}^i]\hat{\pi}^i\,du^i = \gamma^i\int[\log \pi^i_*-\log\hat{\pi}^i]\hat{\pi}^i\,du^i = -\gamma^i D_{\mathrm{KL}}(\hat{\pi}^i\|\pi^i_*)$, so
\begin{align*}
\gamma^i D_{\mathrm{KL}}(\hat{\pi}^i \| \pi^i_*) \le \gamma^i(e^\delta - 1)^2 \le 16\gamma^i(N-1)^2 \varepsilon^2 |U|^4
\end{align*}
for $\delta \le 1$.  Taking the supremum over $(t,x,u^{-i})$ gives the stated bound for $\varepsilon^i_{\mathrm{loc}}$.
\end{proof}

The conditional KL estimate also controls the loss in the corresponding frozen control problem.

\begin{corollary}\label{cor:value_diff}
Under the hypotheses of Theorem~\ref{thm:approx_CE}, assume in addition that the constructed conditional policy $\hat\pi^i(\cdot\mid u^{-i};t,x)$ is admissible for the frozen control problem. Let $\tilde V^i(t,x;u^{-i})$ be the frozen-opponent optimal value and let $\tilde V^i_{\hat\pi}(t,x;u^{-i})$ denote the value under $\hat\pi^i$ with the same frozen $u^{-i}$. Then
\begin{align*}
0 \le \tilde{V}^i(t,x;u^{-i}) - \tilde{V}^i_{\hat{\pi}}(t,x;u^{-i}) \le \frac{1 - e^{-\beta^i(T-t)}}{\beta^i}\,\varepsilon^i_{\mathrm{loc}} \le \frac{1}{\beta^i}\,\varepsilon^i_{\mathrm{loc}}, \quad \text{if } \beta^i > 0,
\end{align*}
with the convention that for $\beta^i = 0$ the bound is interpreted in the limiting sense, giving $\tilde{V}^i - \tilde{V}^i_{\hat{\pi}} \le (T-t)\,\varepsilon^i_{\mathrm{loc}}$.
In particular, if the global compatibility-gap bound satisfies $\varepsilon=O(1/\gamma)$ and $\gamma^i=O(\gamma)$, then the value-function loss is $O(1/\gamma)$.
\end{corollary}
\begin{proof}
The value under the approximate policy is the regularized performance functional
\begin{align}\label{eq:Vhat-def}
\tilde V^i_{\hat\pi}(t,x;u^{-i})
&:=\tilde J^i(t,x;\hat\pi^i,u^{-i})\nonumber\\
&=\mathbb E^{\hat\pi}\bigg[e^{-\beta^i(T-t)}g^i(X_T^{\hat\pi})\\
&\quad+\int_t^T e^{-\beta^i(s-t)}\int_U
\big(f^i(s,X_s^{\hat\pi},u^i)-\gamma^i\log\hat\pi^i(u^i\mid s,X_s^{\hat\pi};u^{-i})\big)\nonumber\\
&\hspace{35mm}\times\hat\pi^i(u^i\mid s,X_s^{\hat\pi};u^{-i})\,du^i\,ds\bigg].
\end{align}
where $X^{\hat\pi}$ is the state process driven by $(\hat\pi^i,u^{-i})$. Both values have terminal condition $g^i$.

By definition, $q^i_*=\partial_t\tilde V^i+H^i-\beta^i\tilde V^i$. Proposition~\ref{prop:q_star_normalization} gives the exact Gibbs normalization $\pi_*^i=\exp\{q_*^i/\gamma^i\}$, and therefore
\begin{equation}\label{eq:q-gibbs}
q^i_* = \gamma^i \log \pi^i_*,
\qquad\text{equivalently,}\qquad
H^i = \gamma^i \log \pi^i_* - \partial_t \tilde{V}^i + \beta^i \tilde{V}^i.
\end{equation}
Since $-\partial_t \tilde{V}^i + \beta^i \tilde{V}^i$ does not depend on $u^i$,
substituting \eqref{eq:q-gibbs} yields
\begin{align}
\partial_t \tilde{V}^i - \beta^i \tilde{V}^i
+ \int_U \Bigl[H^i - \gamma^i \log \hat{\pi}^i\Bigr]\,\hat{\pi}^i\, du^i
&= \gamma^i \!\int_U \log\!\frac{\pi^i_*}{\hat{\pi}^i}\;\hat{\pi}^i\, du^i \notag\\[4pt]
&= -\,\gamma^i\, D_{\mathrm{KL}}\!\Bigl(\hat{\pi}^i \,\Big\|\, \pi^i_*\Bigr).
\label{eq:key-id}
\end{align}

Applying It\^o's formula to $e^{-\beta^i(s-t)}\tilde V^i(s,X_s^{\hat\pi};u^{-i})$ under $(\hat\pi^i,u^{-i})$ gives the value-gap identity below. The infinitesimal generator satisfies
\[
\mathcal{L}^{\hat{\pi}} \tilde{V}^i
= \int_U H^i\,\hat{\pi}^i\,du^i-\int_U f^i\,\hat{\pi}^i\,du^i,
\]
because $H^i$ contains the drift, diffusion, and running-cost terms evaluated at $\nabla_x\tilde V^i$ and $D_x^2\tilde V^i$. Combining this identity with \eqref{eq:key-id}, the drift becomes
\begin{equation*}
\partial_t \tilde{V}^i - \beta^i \tilde{V}^i + \mathcal{L}^{\hat{\pi}} \tilde{V}^i
= -\,\gamma^i\, D_{\mathrm{KL}}\!\Bigl(\hat{\pi}^i \,\Big\|\, \pi^i_*\Bigr)
+ \int_U \Bigl(\gamma^i \log \hat{\pi}^i - f^i\Bigr)\,\hat{\pi}^i\, du^i.
\end{equation*}
Integrating from $t$ to $T$ and taking expectations eliminates the stochastic integral. Using the common terminal condition gives
\begin{align*}
\tilde{V}^i(t,x;\,u^{-i})
=& \underbrace{\mathbb{E}^{\hat{\pi}}\!\left[
e^{-\beta^i(T-t)}\, g^i\!\bigl(X_T^{\hat{\pi}}\bigr)
+ \int_t^T e^{-\beta^i(s-t)} \!\int_U
\bigl(f^i - \gamma^i \log \hat{\pi}^i\bigr)\,\hat{\pi}^i\, du^i \, ds
\right]}_{\displaystyle = \;\tilde{V}^i_{\hat{\pi}}(t,x;\,u^{-i})
\quad\text{by \eqref{eq:Vhat-def}}}
\\&+ \mathbb{E}^{\hat{\pi}}\!\left[\int_t^T e^{-\beta^i(s-t)}\,\gamma^i\,
D_{\mathrm{KL}}\!\Bigl(\hat{\pi}^i \,\Big\|\, \pi^i_*\Bigr)\, ds\right].
\end{align*}
Rearranging gives the value-gap representation
\begin{align}\label{eq:gap}
\tilde{V}&^i(t,x;\,u^{-i}) - \tilde{V}^i_{\hat{\pi}}(t,x;\,u^{-i})
\nonumber\\=& \mathbb{E}^{\hat{\pi}}\!\left[\int_t^T e^{-\beta^i(s-t)}\,\gamma^i\,
D_{\mathrm{KL}}\!\Bigl(
\hat{\pi}^i\!\bigl(\cdot\,\big|\,u^{-i};\,s,\,X_s^{\hat{\pi}}\bigr)
\,\Big\|\,
\pi^i_*\!\bigl(\cdot\,\big|\,u^{-i};\,s,\,X_s^{\hat{\pi}}\bigr)
\Bigr)\, ds\right]
\geq\, 0.
\end{align}
Theorem~\ref{thm:approx_CE} bounds the integrand in \eqref{eq:gap} by $\varepsilon^i_{\mathrm{loc}}$, and hence
\begin{equation*}
0\le \tilde V^i(t,x;u^{-i})-\tilde V^i_{\hat\pi}(t,x;u^{-i})
\le \varepsilon^i_{\mathrm{loc}}\int_t^T e^{-\beta^i(s-t)}ds
=\frac{1-e^{-\beta^i(T-t)}}{\beta^i}\varepsilon^i_{\mathrm{loc}}.
\end{equation*}
\end{proof}
Theorem~\ref{thm:approx_CE} controls the discrepancy between the frozen Gibbs conditionals and a compatible joint realization, while Corollary~\ref{cor:value_diff} converts that discrepancy into a value loss for the frozen control problem. Neither result concerns the value gap in the standard $N$-player Nash problem. If the compatibility-gap estimate is available only on $[0,T]\times B_R\times U^N$, the coordinate-path construction and conditional KL bound remain valid there, but a dynamic value bound requires either control along the whole state trajectory or stopping at the first exit from $B_R$.

\section{Stationary Discounted Frozen Responses}\label{sec:ergodic}

Consider the time-homogeneous discounted counterpart of the frozen-opponent problem. Since $\beta^i>0$, this is an infinite-horizon discounted problem rather than an ergodic average-reward problem. We write $V^i(x;u^{-i})$ for the stationary frozen value.

\subsection{Problem Setup}

Consider an infinite-horizon game with time-homogeneous coefficients $f^i = f^i(x, u^i)$, $b^i = b^i(x, u)$, $\sigma^i = \sigma^i(x, u)$, and $U$ compact.  Let $\Pi_{\rm stat}$ denote stationary Markov densities $\pi(u;x)$ satisfying the stationary analogues of Definition~\ref{def:admissible_policy} and having finite discounted performance. The state process evolves as
\begin{align*}
dX^i_s = b^i(X_s, \mathbf{u}_s)ds + \sigma^i(X_s, \mathbf{u}_s)dW^i_s.
\end{align*}
Each player discounts at rate $\beta^i > 0$. The stationary discounted frozen value is
\begin{align*}
V^i(x;u^{-i}) := \sup_{\pi^i\in\Pi_{\rm stat}}\mathbb E\bigg[\int_0^\infty e^{-\beta^is}\int_U
&\Big[f^i(\tilde X_s^{0,x,\pi,u^{-i}},u^i)
-\gamma^i\log\pi^i(u^i;\tilde X_s^{0,x,\pi,u^{-i}})\Big]\\
&\times\pi^i(u^i;\tilde X_s^{0,x,\pi,u^{-i}})\,du^i\,ds\bigg].
\end{align*}
Time-homogeneity ensures $V^i = V^i(x; u^{-i})$ has no explicit time dependence, so the HJB equation is elliptic.

\subsection{Stationary HJB Equation and Compatibility}

Under the stated classical regularity, the stationary HJB equation is
\begin{align}\label{eq:ergodic_HJB}
\beta^i V^i(x; u^{-i}) = G^i(\nabla_x V^i, D^2_x V^i; u^{-i}, \gamma^i),
\end{align}
where
$G^i(p,A;u^{-i},\gamma^i):=\gamma^i\log\int_U\exp\{H^i(x,(u^i,u^{-i}),p,A)/\gamma^i\}du^i$,
with the time argument suppressed.  The conditional optimal policy is $\pi^i_*(u^i\mid x,u^{-i})=\exp\{q^i_*(x,u)/\gamma^i\}$, where
$q^i_*(x,u):=H^i(x,(u^i,u^{-i}),\nabla_xV^i,D_x^2V^i)-\beta^iV^i(x;u^{-i})$ and $\int_Ue^{q^i_*/\gamma^i}du^i=1$. In the results below we assume that this elliptic equation has the stated classical solution and local Schauder bounds. Standard sufficient conditions include uniform ellipticity together with appropriate dissipativity, growth, and H\"older assumptions; see \cite{Krylov87,GT01}.

The cross-partial decomposition has exactly the same form as in the finite-horizon frozen problem.

\begin{proposition}\label{prop:ergodic_decomp}
Under Assumption~\ref{assumption regularity} with time-homogeneous coefficients, the decomposition~\eqref{eq:cross_partial_decomp} holds with $\tilde{V}^i$ replaced by $V^i$ and all time derivatives removed.
\end{proposition}

\begin{proof}
For the stationary frozen problem,
\[
q_*^i(x,u)=H^i(x,u,\partial_xV^i,\partial_x^2V^i)-\beta^iV^i(x;u^{-i}).
\]
Since $V^i(x;u^{-i})$ is independent of the own action $u^i$,
\[
\partial_{u^i}q_*^i=\partial_{u^i}H^i(x,u,\partial_xV^i,\partial_x^2V^i).
\]
Differentiate with respect to $u^j$, $j\ne i$.  The product and chain rules give
\begin{align*}
\partial_{u^j}\partial_{u^i}q_*^i
={}&\partial_{u^j}\partial_{u^i}H^i
+\sum_{k=1}^N\partial_{u^i p^k}H^i\,\partial_{x^k}V^i_{u^j}
+\sum_{k=1}^N\partial_{u^i A^{k,k}}H^i\,\partial_{x^kx^k}^2V^i_{u^j},
\end{align*}
where all derivatives of $H^i$ are evaluated at $(x,u,\partial_xV^i,\partial_x^2V^i)$.  Inserting the definition~\eqref{eq Hamiltonian} yields exactly the direct and indirect terms in~\eqref{eq:cross_partial_decomp}; the only missing terms relative to the finite-horizon formula are those involving time derivatives, which are absent here.
\end{proof}

\subsection{Large-\texorpdfstring{$\gamma$}{gamma} and Approximate Joint Realization}

The finite-horizon large-exploration argument carries over to the stationary problem after removing the time derivatives.

\begin{proposition}\label{prop:ergodic_large_gamma}
Suppose $U\subset\mathbb R$ is compact.  Assume the stationary frozen HJB has a classical solution and, for each $R>0$, the stationary analog of~\eqref{eq:uniform_frozen_jets} is bounded by a constant $C_R$ independent of sufficiently large $\gamma^i$.  Then
\begin{align*}
\sup_{B_R\times U^N}\left|\pi_*^i(u^i\mid x,u^{-i})-\frac1{|U|}\right|&=O_R((\gamma^i)^{-1}),\\
\max_{i\ne j}\sup_{B_R\times U^N}|\Delta^{ij}(x,u)|&=O_R(\gamma^{-1}),
\qquad \gamma:=\min_k\gamma^k.
\end{align*}
\end{proposition}

\begin{proof}
The assumed stationary jet bounds and compactness of $U$ give a constant $M_R$, independent of sufficiently large $\gamma^i$, such that
\[
|H^i(x,u,\partial_xV^i,D_x^2V^i)|\le M_R
\quad\text{on }B_R\times U^N.
\]
The stationary Gibbs formula then gives the same two-sided estimate
\[
\frac{e^{-2M_R/\gamma^i}}{|U|}\le \pi_*^i(u^i\mid x,u^{-i})
\le\frac{e^{2M_R/\gamma^i}}{|U|},
\]
so the policy differs from the uniform density by $O_R((\gamma^i)^{-1})$.

For $j\ne i$, Proposition~\ref{prop:ergodic_decomp} and the stationary sensitivity bounds imply
\[
\sup_{B_R\times U^N}|\partial_{u^j}\partial_{u^i}q_*^i|\le C_R.
\]
Therefore the compatibility gap satisfies
\[
|\Delta^{ij}(x,u)|
\le C_R((\gamma^i)^{-1}+(\gamma^j)^{-1})
\le 2C_R/\gamma,
\]
which proves the claim.
\end{proof}

The approximate joint-realization argument applies pointwise in $x$.

\begin{corollary}\label{cor:ergodic_CE}
Under the hypotheses of Proposition~\ref{prop:ergodic_large_gamma}, suppose the stationary compatibility gaps satisfy
\[
\sup_{(x,u)\in\mathbb R^N\times U^N}|\Delta^{ij}(x,u)|\le\varepsilon,
\qquad i\ne j.
\]
Then, for each $x\in\mathbb R^N$, there exists $\hat\psi(\cdot;x)\in\cP(U^N)$ with the following properties.
\begin{enumerate}
\item $D_{\mathrm{KL}}(\hat{\pi}^i(\cdot|u^{-i}; x) \| \pi^i_*(\cdot|u^{-i}; x)) \le (e^\delta - 1)^2$, where $\delta = 2(N-1)\varepsilon |U|^2$.
\item With $\varepsilon^i_{\mathrm{loc}}:=\sup_{(x,u^{-i})}\gamma^iD_{\mathrm{KL}}(\hat\pi^i(\cdot\mid x,u^{-i})\|\pi_*^i(\cdot\mid x,u^{-i}))$, one has $\varepsilon^i_{\mathrm{loc}} \le \gamma^i(e^\delta - 1)^2 \le 16\gamma^i(N-1)^2 \varepsilon^2 |U|^4$ for $\delta \le 1$.
\item If the constructed conditional policy $\hat\pi^i(\cdot\mid x,u^{-i})$ is admissible for the stationary frozen problem and
$\lim_{T\to\infty}\mathbb E^{\hat\pi}[e^{-\beta^iT}(|V^i(X_T;u^{-i})|+|V^i_{\hat\pi}(X_T;u^{-i})|)]=0$, then
$0\le V^i(x;u^{-i})-V^i_{\hat\pi}(x;u^{-i})\le\varepsilon^i_{\mathrm{loc}}/\beta^i$.
\end{enumerate}
If only a local gap bound on $B_R\times U^N$ is available, conclusions (1)--(2) hold locally.  A dynamic value bound then requires either a global gap estimate or stopping at the first exit from $B_R$.
\end{corollary}

\begin{proof}
Fix $x$.  Apply the coordinate-path construction in the proof of Theorem~\ref{thm:approx_CE} to the stationary score field
\[
a^k(u)=\frac1{\gamma^k}\partial_{u^k}q_*^k(x,u).
\]
The bound $|\Delta^{ij}|\le\varepsilon$ gives
\[
|\partial_{u^k}\Phi-a^k|\le(N-1)\varepsilon|U|,
\qquad
|\log(\pi_*^i/\hat\pi^i)|\le\delta,
\quad \delta=2(N-1)\varepsilon|U|^2.
\]
The $\chi^2$ estimate used there yields
\[
D_{\rm KL}(\hat\pi^i\|\pi_*^i)\le(e^\delta-1)^2,
\]
and the entropy variational identity gives the pointwise loss
\[
\gamma^iD_{\rm KL}(\hat\pi^i\|\pi_*^i).
\]
Taking the supremum over $(x,u^{-i})$ gives $\varepsilon^i_{\rm loc}$ as defined in the corollary.

For the value bound, use the stationary identity
\[
q_*^i(x,u)=H^i(x,u,\partial_xV^i,D_x^2V^i)-\beta^iV^i(x;u^{-i})
=\gamma^i\log\pi_*^i(u^i\mid x,u^{-i}).
\]
Apply It\^o's formula to $e^{-\beta^is}V^i(X_s^{\hat\pi};u^{-i})$ under the frozen policy $\hat\pi^i$.  After taking expectations on $[0,T]$, the same calculation as in~\eqref{eq:gap} gives
\begin{align*}
V^i(x;u^{-i})-V^i_{\hat\pi}(x;u^{-i})
={}&\mathbb E^{\hat\pi}\int_0^T e^{-\beta^is}\gamma^iD_{\rm KL}
(\hat\pi^i(\cdot\mid X_s,u^{-i})\|\pi_*^i(\cdot\mid X_s,u^{-i}))ds\\
&+\mathbb E^{\hat\pi}\big[e^{-\beta^iT}
(V^i(X_T;u^{-i})-V^i_{\hat\pi}(X_T;u^{-i}))\big].
\end{align*}
The last expectation tends to zero by the transversality assumption in the corollary.  Letting $T\to\infty$ and using the uniform KL bound gives
\[
0\le V^i(x;u^{-i})-V^i_{\hat\pi}(x;u^{-i})
\le\varepsilon^i_{\rm loc}\int_0^\infty e^{-\beta^is}ds
=\frac{\varepsilon^i_{\rm loc}}{\beta^i}.
\]
The local statement for conclusions (1)--(2) follows by applying the same construction on the stated compact region.
\end{proof}

\section{\texorpdfstring{$q$}{q}-Learning for the Frozen-Opponent Problem}
\label{sec:qlearning}

\subsection{The \texorpdfstring{$q$}{q}-Function}

Fix player $i$ and treat $u^{-i}$ as a deterministic frozen parameter.  The drift and diffusion may depend on the full control vector $u=(u^i,u^{-i})$.

For the martingale results we use the discretely sampled action process from the erratum to \cite{JZ23}.  Given a grid $\mathcal G_{t:T}=\{t=s_0<s_1<\cdots<s_M=T\}$ and independent random seeds $\xi_0,\ldots,\xi_{M-1}$, choose a measurable sampling map $\varphi$ such that $\varphi(s,x,\xi_k)$ has density $\pi(\cdot;s,x)$.  On $[s_k,s_{k+1})$ set
\[
a_s^{\mathcal G,\pi}=\varphi(s_k,X_{s_k}^{\mathcal G,\pi},\xi_k),
\]
and let $X^{t,x,\mathcal G,\pi,u^{-i}}$ solve the original state equation with player $i$'s action equal to $a^{\mathcal G,\pi}$ and the opponents frozen at $u^{-i}$.  The state remains continuous while the action is sampled only at grid points.  This construction avoids the progressive-measurability issue of continuum independent sampling \cite{JZ23err}.

The quantity learned below is the instantaneous advantage rate associated with a fixed policy in this frozen problem.

\begin{definition}\label{def:q_function}
For $u^{-i}\in U^{N-1}$ fixed and $\pi\in\Pi$, the \emph{$q^i$-function} associated with $\pi$ is defined by
\begin{align}\label{eq:q_def}
q^i&(t,x,u^i,u^{-i};\pi) := \partial_t\tilde{J}^i(t,x;\pi,u^{-i})\nonumber\\
&+ H^i\big(t,x,u^i,u^{-i},\partial_x\tilde{J}^i(t,x;\pi,u^{-i}),\partial_x^2\tilde{J}^i(t,x;\pi,u^{-i})\big) - \beta^i\tilde{J}^i(t,x;\pi,u^{-i}),
\end{align}
where $\tilde{J}^i(t,x;\pi,u^{-i})$ is the performance functional of player $i$ under fixed policy $\pi^i=\pi$ against the frozen opponents' profile $u^{-i}$, and $H^i$ is the conditional Hamiltonian~\eqref{eq Hamiltonian} evaluated at the spatial derivatives of $\tilde{J}^i$ and the action profile $u=(u^i,u^{-i})$.
\end{definition}

The fixed-policy evaluation equation gives the normalization, centering, continuity, and policy-improvement identities collected below.

\begin{proposition}\label{prop:q_properties}
Let $\pi\in\Pi$ be admissible and let $q^i$ be given by \eqref{eq:q_def}. Then the following properties hold.
\begin{enumerate}
\item[(i)] \emph{(Normalization)} For all $(t,x,u^{-i})\in[0,T]\times\mathbb{R}^N\times U^{N-1}$,
\begin{align*}
\int_U\big[q^i(t,x,u^i,u^{-i};\pi) - \gamma^i\log\pi(u^i;t,x)\big]\pi(u^i;t,x)du^i = 0.
\end{align*}
\item[(ii)] \emph{(Centering identity)} Writing
\[
\overline H^{i,\pi}(t,x,u^{-i})
:=\int_U\big[H^i(t,x,z^i,u^{-i},\partial_x\tilde J^i,\partial_x^2\tilde J^i)-\gamma^i\log\pi(z^i;t,x)\big]\pi(z^i;t,x)dz^i,
\]
one has
\begin{align}\label{eq:q_PDE}
q^i(t,x,u^i,u^{-i};\pi)
=H^i(t,x,u^i,u^{-i},\partial_x\tilde J^i,\partial_x^2\tilde J^i)-\overline H^{i,\pi}(t,x,u^{-i}).
\end{align}
\item[(iii)] \emph{(Continuity)} If $H^i$ is continuous in $(t,x,u)$ and the derivatives appearing in~\eqref{eq:q_def} are jointly continuous in $(t,x,u^{-i})$, then $q^i$ is continuous in $(t,x,u^i,u^{-i})$.
\item[(iv)] \emph{(Policy Improvement Representation)} The improved policy satisfies
\[
\pi'(u^i;t,x,u^{-i})\propto\exp\{q^i(t,x,u^i,u^{-i};\pi)/\gamma^i\}.
\]
\end{enumerate}
\end{proposition}

\begin{proof}
(i) For fixed $\pi$, the evaluation equation is
\[
\partial_t\tilde J^i-\beta^i\tilde J^i
+\int_U\big[H^i(t,x,u^i,u^{-i},\partial_x\tilde J^i,\partial_x^2\tilde J^i)-\gamma^i\log\pi(u^i;t,x)\big]\pi(u^i;t,x)du^i=0.
\]
Insert the definition~\eqref{eq:q_def}.  Since $\partial_t\tilde J^i-\beta^i\tilde J^i$ is independent of $u^i$ and $\int_U\pi\,du^i=1$, the left-hand side becomes $\int_U[q^i-\gamma^i\log\pi]\pi\,du^i$, proving (i).

(ii) The same evaluation equation gives
\[
\partial_t\tilde J^i-\beta^i\tilde J^i=-\overline H^{i,\pi}.
\]
Substitution into~\eqref{eq:q_def} yields~\eqref{eq:q_PDE}.

(iii) This follows directly from~\eqref{eq:q_def} and the stated joint continuity.

(iv) From \eqref{eq:q_def}, $H^i = q^i - \partial_t\tilde{J}^i + \beta^i\tilde{J}^i$, where $-\partial_t\tilde{J}^i + \beta^i\tilde{J}^i$ is independent of $u^i$. Hence $\exp\{H^i/\gamma^i\}=\exp\{q^i/\gamma^i\}\cdot\exp\{(-\partial_t\tilde{J}^i+\beta^i\tilde{J}^i)/\gamma^i\}$, and the second factor cancels after normalization, yielding $\pi'\propto\exp\{q^i/\gamma^i\}$.
\end{proof}

\subsection{Policy Improvement}

For fixed $u^{-i}$, exponentiating the current $q^i$-function gives the usual entropy-regularized policy update. The update improves the frozen value whenever it remains admissible.

\begin{theorem}\label{thm policy improvement}
For any $\pi\in\Pi$, define the improved policy
\[
\pi'(\cdot;t,x,u^{-i})\propto\exp\bigg\{\frac{1}{\gamma^i}H^i\big(t,x,\cdot,u^{-i},\partial_x\tilde{J}^i(t,x;\pi,u^{-i}),\partial_x^2\tilde{J}^i(t,x;\pi,u^{-i})\big)\bigg\}.
\]
If, for each frozen parameter $u^{-i}$, the resulting Markov policy $\pi'(\cdot;\,u^{-i})$ belongs to $\Pi$, then $\tilde{J}^i(t,x;\pi',u^{-i})\geq\tilde{J}^i(t,x;\pi,u^{-i})$ for all $(t,x,u^{-i})$.

For a fixed frozen parameter $u^{-i}$, if $\mathcal{I}^i_{u^{-i}}(\pi)=\pi$, where
\begin{align*}
\mathcal{I}^i_{u^{-i}}(\pi)(u^i;t,x) := \frac{\exp\big\{\frac{1}{\gamma^i}H^i(t,x,u^i,u^{-i},\partial_x\tilde{J}^i,\partial_x^2\tilde{J}^i)\big\}}{\int_U\exp\big\{\frac{1}{\gamma^i}H^i(t,x,u^i,u^{-i},\partial_x\tilde{J}^i,\partial_x^2\tilde{J}^i)\big\}du^i},
\end{align*}
then $\pi$ is the optimal policy $\pi^i_*$.
\end{theorem}

\begin{proof}
The argument adapts the single-agent policy improvement proof of \cite[Lemma~13]{JZ23} to the conditional setting. Apply It\^o's formula to $e^{-\beta^i s}\tilde{J}^i(s,X_s^{t,x,\pi'};\pi,u^{-i})$ from $t$ to $T$. The key inequality is: for any $(s,y)$,
\begin{align*}
\int_U\big[&q^i(s,y,u^i,u^{-i};\pi) - \gamma^i\log\pi'(u^i;s,y)\big]\pi'(u^i;s,y)du^i \\
&\geq \int_U\big[q^i(s,y,u^i,u^{-i};\pi) - \gamma^i\log\pi(u^i;s,y)\big]\pi(u^i;s,y)du^i = 0,
\end{align*}
since $\pi'=\mathcal{I}^i_{u^{-i}}(\pi)$ is the unique maximizer of $\pi'\mapsto\int_U[q^i(\cdot,u^i,u^{-i};\pi)-\gamma^i\log\pi']\pi'du^i$ by the classical entropy optimization result (\cite{JZ23}, Lemma~13). Taking expectations and sending the localizing sequence $R\to\infty$ via stopping times $\tau_R := \inf\{s : |X_s| \ge R\}$ yields $\tilde{J}^i(t,x;\pi,u^{-i})\leq\tilde{J}^i(t,x;\pi',u^{-i})$.

For the fixed-point statement, fix $u^{-i}$.  If $\mathcal{I}^i_{u^{-i}}(\pi)=\pi$, then $\pi$ achieves the supremum in the frozen HJB~\eqref{eq HJB-explore}.  Hence $\tilde J^i(\cdot;\pi,u^{-i})$ solves~\eqref{eq explore HJB 2}; uniqueness in the stated classical solution class gives $\tilde J^i=\tilde V^i$ and $\pi=\pi_*^i$.
\end{proof}

Equivalently, the policy update can be written directly in terms of the $q^i$-function as
\begin{align}\label{eq:improved_policy_q}
\pi'(\cdot;t,x,u^{-i})\propto\exp\bigg\{\frac{1}{\gamma^i}q^i(t,x,\cdot,u^{-i};\pi)\bigg\}.
\end{align}

\subsection{Weak Martingale Characterization}

The $q^i$-function has the following martingale characterization.

\begin{theorem}\label{thm martingale represent 1}
Let $\pi\in\Pi$, and let $\hat q^i:[0,T]\times\mathbb R^N\times U^N\to\mathbb R$ be continuous.  Then $\hat q^i=q^i(\cdot;\pi)$ if and only if, for every $(t,x,u^{-i})$ and every time grid $\mathcal G_{t:T}$, the process
\begin{align}\label{eq:martingale_cond}
M_s^{\mathcal G}:={}&e^{-\beta^i s}\tilde J^i(s,X_s^{t,x,\mathcal G,\pi,u^{-i}};\pi,u^{-i})\\
&+\int_t^s e^{-\beta^i r}\big[f^i(r,X_r^{t,x,\mathcal G,\pi,u^{-i}},a_r^{\mathcal G,\pi})
-\hat q^i(r,X_r^{t,x,\mathcal G,\pi,u^{-i}},a_r^{\mathcal G,\pi},u^{-i})\big]dr
\end{align}
is a martingale.  In either case,
\begin{align}\label{eq:q_normalization}
\int_U\big[q^i(t,x,u^i,u^{-i};\pi)-\gamma^i\log\pi(u^i;t,x)\big]\pi(u^i;t,x)du^i=0.
\end{align}
\end{theorem}

\begin{proof}
Assume first that $\hat q^i=q^i$.  On each grid interval $[s_k,s_{k+1})$ the sampled action $a^{\mathcal G,\pi}$ is constant and adapted.  It\^o's formula therefore applies to $e^{-\beta^i r}\tilde J^i(r,X_r^{t,x,\mathcal G,\pi,u^{-i}};\pi,u^{-i})$.  Suppressing the fixed arguments $(\pi,u^{-i})$ and writing $X_r=X_r^{t,x,\mathcal G,\pi,u^{-i}}$, we obtain
\begin{align*}
&d\big(e^{-\beta^i r}\tilde J^i(r,X_r)\big)
+e^{-\beta^i r}\big[f^i(r,X_r,a_r^{\mathcal G,\pi})-q^i(r,X_r,a_r^{\mathcal G,\pi},u^{-i};\pi)\big]dr\\
&\quad=e^{-\beta^i r}\Big[\partial_t\tilde J^i
+H^i(r,X_r,a_r^{\mathcal G,\pi},u^{-i},\partial_x\tilde J^i,\partial_x^2\tilde J^i)
-\beta^i\tilde J^i-q^i(r,X_r,a_r^{\mathcal G,\pi},u^{-i};\pi)\Big]dr\\
&\qquad\quad+e^{-\beta^i r}\partial_x\tilde J^i(r,X_r)\sigma(r,X_r,a_r^{\mathcal G,\pi},u^{-i})\,dW_r.
\end{align*}
The drift is zero by~\eqref{eq:q_def}.  After localization, the stochastic integral is a martingale; the imposed growth assumptions permit removal of the localization.  Integrating from $t$ to $s$ proves the martingale property of~\eqref{eq:martingale_cond}.

Conversely, assume~\eqref{eq:martingale_cond} is a martingale for every grid.  Applying It\^o's formula as above, now with $\hat q^i$, shows that
\[
A_s:=\int_t^s e^{-\beta^i r}\big[q^i(r,X_r,a_r^{\mathcal G,\pi},u^{-i};\pi)-\hat q^i(r,X_r,a_r^{\mathcal G,\pi},u^{-i})\big]dr
\]
is both a continuous finite-variation process and a local martingale.  Hence $A_s=0$ a.s. for every $s$.

Set $F=q^i-\hat q^i$.  If $F(t_0,x_0,a_0,u^{-i})>0$, continuity gives $h>0$, neighborhoods $B_x$ of $x_0$ and $B_a$ of $a_0$, and $c>0$ such that $F(r,y,a,u^{-i})\ge c$ for $r\in[t_0,t_0+h]$, $y\in B_x$, and $a\in B_a$.  Choose a grid whose first interval is $[t_0,t_0+h]$.  At $t_0$ the first sampled action has density $\pi(\cdot;t_0,x_0)$; full support gives
\[
\mathbb P(a_{t_0}^{\mathcal G,\pi}\in B_a)>0.
\]
On that event the same action is held over the first grid interval.  Conditional on any $a\in B_a$, continuity of the SDE solution and local boundedness of the coefficients imply, after reducing $h$ if necessary,
\[
\mathbb P\big(X_r\in B_x\text{ for all }r\in[t_0,t_0+h]\mid a_{t_0}^{\mathcal G,\pi}=a\big)>0.
\]
Thus with positive probability $F(r,X_r,a_r^{\mathcal G,\pi},u^{-i})\ge c$ throughout the first interval, and $A_{t_0+h}>0$, contradicting $A\equiv0$.  The case $F(t_0,x_0,a_0,u^{-i})<0$ is identical with the signs reversed.  Therefore $F\equiv0$.

Finally,~\eqref{eq:q_normalization} follows from the fixed-policy evaluation equation exactly as in Proposition~\ref{prop:q_properties}(i).
\end{proof}

If both the value function and the $q^i$-function are unknown, the terminal condition and the normalization identity identify them jointly.

\begin{corollary}\label{coro martingale representation 2}
For each fixed $u^{-i}$, let $\hat J^i(\cdot,\cdot;u^{-i})\in C^{1,2}([0,T]\times\mathbb R^N)$ have polynomial growth, and let $\hat q^i$ be continuous. Suppose
\[
\hat J^i(T,x,u^{-i})=g^i(x),\qquad
\int_U[\hat q^i(t,x,u^i,u^{-i})-\gamma^i\log\pi(u^i;t,x)]\pi(u^i;t,x)du^i=0.
\]
Then $\hat J^i=\tilde J^i(\cdot;\pi,\cdot)$ and $\hat q^i=q^i(\cdot;\pi)$ if and only if, for every $(t,x,u^{-i})$ and every grid $\mathcal G_{t:T}$, the process obtained from~\eqref{eq:martingale_cond} by replacing $(\tilde J^i,q^i)$ with $(\hat J^i,\hat q^i)$ is a martingale.
\end{corollary}

\begin{proof}
If $(\hat J^i,\hat q^i)=(\tilde J^i,q^i)$, the claim is Theorem~\ref{thm martingale represent 1}.

For the converse define
\[
q_{\hat J}^i(t,x,u):=\partial_t\hat J^i(t,x,u^{-i})
+H^i(t,x,u,\partial_x\hat J^i,\partial_x^2\hat J^i)-\beta^i\hat J^i(t,x,u^{-i}).
\]
Applying It\^o's formula to $e^{-\beta^i s}\hat J^i(s,X_s^{\mathcal G,\pi})$ and subtracting the assumed martingale shows that
\[
\int_t^s e^{-\beta^i r}\big[q_{\hat J}^i(r,X_r^{\mathcal G,\pi},a_r^{\mathcal G,\pi},u^{-i})-
\hat q^i(r,X_r^{\mathcal G,\pi},a_r^{\mathcal G,\pi},u^{-i})\big]dr=0
\]
a.s. for every $s$ and every grid.  If the continuous function $q_{\hat J}^i-\hat q^i$ were positive or negative at some $(t_0,x_0,a_0,u^{-i})$, choose a first grid cell $[t_0,t_0+h]$, sample an action in a small neighborhood of $a_0$, and use continuity of the state to keep it in a small neighborhood of $x_0$ during that cell.  This event has positive probability by full support of $\pi$, and on it the displayed integral has a strict sign, a contradiction.  Hence $\hat q^i=q_{\hat J}^i$ pointwise. Substituting this identity into the normalization constraint and using the fact that $\partial_t\hat J^i-\beta^i\hat J^i$ is independent of $u^i$ gives
\begin{align*}
0={}&\partial_t\hat J^i(t,x,u^{-i})-\beta^i\hat J^i(t,x,u^{-i})\\
&+\int_U\big[H^i(t,x,u^i,u^{-i},\partial_x\hat J^i,\partial_x^2\hat J^i)-\gamma^i\log\pi(u^i;t,x)\big]\pi(u^i;t,x)du^i.
\end{align*}
Together with $\hat J^i(T,x,u^{-i})=g^i(x)$, this is precisely the fixed-policy evaluation PDE for $\pi$.  Uniqueness in the stated regularity class gives $\hat J^i=\tilde J^i(\cdot;\pi,u^{-i})$, and the pointwise identity above then gives $\hat q^i=q^i(\cdot;\pi)$.
\end{proof}

\subsection{Optimal \texorpdfstring{$q^i$}{q}-Function}

For the optimal frozen policy $\pi_*^i$, the associated $q^i$-function is
\begin{align}\label{eq q*}
q^i_*(t,x,u^i,u^{-i})& := \partial_t\tilde{V}^i(t,x;u^{-i})\nonumber\\
&+ H^i(t,x,u^i,u^{-i},\partial_x\tilde{V}^i(t,x;u^{-i}),\partial_x^2\tilde{V}^i(t,x;u^{-i})) - \beta^i\tilde{V}^i(t,x;u^{-i}).
\end{align}

The HJB equation fixes the additive normalization of the optimal $q^i$-function.

\begin{proposition}\label{prop:q_star_normalization}
For every $(t,x,u^{-i})\in[0,T]\times\mathbb R^N\times U^{N-1}$,
\begin{align*}
\int_U\exp\bigg\{\frac{1}{\gamma^i}q^i_*(t,x,u^i,u^{-i})\bigg\}du^i = 1.
\end{align*}
Consequently, the optimal Gibbs policy requires no additional normalizing factor and is given by
\begin{align*}
\pi^i_*(u^i;t,x,u^{-i}) = \exp\bigg\{\frac{1}{\gamma^i}q^i_*(t,x,u^i,u^{-i})\bigg\}.
\end{align*}
\end{proposition}

\begin{proof}
Using \eqref{eq q*} and the exploratory HJB equation \eqref{eq explore HJB 2},
\begin{align*}
\int_U\exp\bigg\{\frac{1}{\gamma^i}q^i_*\bigg\}du^i &= \int_U\exp\bigg\{\frac{1}{\gamma^i}\big[\partial_t\tilde{V}^i - \beta^i\tilde{V}^i + H^i\big]\bigg\}du^i\\
&= \exp\bigg\{\frac{1}{\gamma^i}[\partial_t\tilde{V}^i - \beta^i\tilde{V}^i]\bigg\}\int_U\exp\bigg\{\frac{1}{\gamma^i}H^i(t,x,u^i,u^{-i},\ldots)\bigg\}du^i.
\end{align*}
By \eqref{eq explore HJB 2}, $\partial_t\tilde{V}^i - \beta^i\tilde{V}^i = -\gamma^i\log\int_U\exp\{\frac{1}{\gamma^i}H^i\}du^i$, so the right-hand side equals $1$.
\end{proof}

This normalization allows the optimal pair to be identified from martingale data generated by any full-support behavior policy.

\begin{theorem}\label{thm martingale represent 3}
For each fixed $u^{-i}$, let $\hat J_*^i(\cdot,\cdot;u^{-i})\in C^{1,2}([0,T]\times\mathbb R^N)$ have polynomial growth, and let $\hat q_*^i$ be continuous. Assume
\[
\hat J^i_*(T,x,u^{-i})=g^i(x),\qquad
\int_U\exp\{\hat q^i_*(t,x,u^i,u^{-i})/\gamma^i\}\,du^i=1.
\]
\begin{enumerate}
\item[(1)] If $(\hat J^i_*,\hat q^i_*)=(\tilde V^i,q^i_*)$, then for every behavior policy $\pi\in\Pi$, every $(t,x,u^{-i})$, and every grid $\mathcal G_{t:T}$,
\begin{align}\label{eq optimal martingale representation}
&e^{-\beta^i s}\hat J^i_*(s,X_s^{t,x,\mathcal G,\pi,u^{-i}},u^{-i})\\
&\quad+\int_t^s e^{-\beta^i r}\big[f^i(r,X_r^{t,x,\mathcal G,\pi,u^{-i}},a_r^{\mathcal G,\pi})
-\hat q^i_*(r,X_r^{t,x,\mathcal G,\pi,u^{-i}},a_r^{\mathcal G,\pi},u^{-i})\big]dr
\end{align}
is a martingale.
\item[(2)] Conversely, if there exists one full-support behavior policy $\pi\in\Pi$ for which~\eqref{eq optimal martingale representation} is a martingale for every $(t,x,u^{-i})$ and every grid, then $(\hat J^i_*,\hat q^i_*)=(\tilde V^i,q^i_*)$.
\end{enumerate}
\end{theorem}

\begin{proof}
For (1), apply It\^o's formula on each grid interval to $e^{-\beta^i r}\tilde V^i(r,X_r^{\mathcal G,\pi};u^{-i})$.  The finite-variation term in~\eqref{eq optimal martingale representation} is
\[
e^{-\beta^i r}\big[\partial_t\tilde V^i+H^i(r,X_r,a_r^{\mathcal G,\pi},u^{-i},\partial_x\tilde V^i,\partial_x^2\tilde V^i)-\beta^i\tilde V^i-q^i_*(r,X_r,a_r^{\mathcal G,\pi},u^{-i})\big]dr,
\]
which vanishes by~\eqref{eq q*}.  After localization and the usual integrability argument, only the stochastic integral remains, proving the martingale property.

For (2), define from the candidate value
\[
q_{\hat J_*}^i:=\partial_t\hat J_*^i
+H^i(t,x,u,\partial_x\hat J_*^i,\partial_x^2\hat J_*^i)-\beta^i\hat J_*^i.
\]
The same It\^o decomposition shows that, for every grid,
\[
\int_t^s e^{-\beta^i r}\big[q_{\hat J_*}^i(r,X_r^{\mathcal G,\pi},a_r^{\mathcal G,\pi},u^{-i})-
\hat q_*^i(r,X_r^{\mathcal G,\pi},a_r^{\mathcal G,\pi},u^{-i})\big]dr=0
\]
a.s.  If $q_{\hat J_*}^i-\hat q_*^i$ had a strict sign at some point, continuity, full support of the first sampled action, and continuity of the state on a sufficiently short first grid cell would produce a positive-probability event on which the integral has that same strict sign.  Hence $\hat q_*^i=q_{\hat J_*}^i$ for all $(t,x,u)$. Using this identity in the normalization condition gives
\begin{align*}
1&=\int_U\exp\Big\{\frac{1}{\gamma^i}\big[\partial_t\hat J_*^i-\beta^i\hat J_*^i+H^i(t,x,u^i,u^{-i},\partial_x\hat J_*^i,\partial_x^2\hat J_*^i)\big]\Big\}du^i\\
&=\exp\{(\partial_t\hat J_*^i-\beta^i\hat J_*^i)/\gamma^i\}
\int_U\exp\{H^i(t,x,u^i,u^{-i},\partial_x\hat J_*^i,\partial_x^2\hat J_*^i)/\gamma^i\}du^i.
\end{align*}
Taking logarithms yields the frozen exploratory HJB~\eqref{eq explore HJB 2} with $\hat J_*^i$ in place of $\tilde V^i$.  The terminal conditions agree, so uniqueness of the classical HJB solution gives $\hat J_*^i=\tilde V^i$.  The pointwise identity $\hat q_*^i=q_{\hat J_*}^i$ then gives $\hat q_*^i=q_*^i$ by~\eqref{eq q*}.
\end{proof}

\section{\texorpdfstring{$q$}{q}-Learning Algorithms}\label{sec:algorithms}
\label{sec:algorithm}

The martingale results above lead to TD estimators for the frozen-response problem.  The sampling protocol must match that problem: $u^{-i}$ is held fixed along a trajectory, while player $i$'s action is sampled on a time grid.

\subsection{Temporal-Difference Condition}

For a grid interval $[s_k,s_{k+1}]$, Theorem~\ref{thm martingale represent 1} gives the increment
\begin{align}\label{eq:TD_residual}
\delta_k^i:={}&e^{-\beta^i s_{k+1}}\hat J^i(s_{k+1},X_{s_{k+1}},u^{-i})
-e^{-\beta^i s_k}\hat J^i(s_k,X_{s_k},u^{-i})\\
&+\int_{s_k}^{s_{k+1}}e^{-\beta^i r}\big[f^i(r,X_r,a_r^{\mathcal G,\pi})
-\hat q^i(r,X_r,a_r^{\mathcal G,\pi},u^{-i})\big]dr.
\end{align}
At the population solution, $\mathbb E[\delta_k^i\mid\mathscr F_{s_k}]=0$.  A parametric approximation $(\hat J^{i,\phi},\hat q^{i,\theta})$ can therefore be fitted by stochastic approximation to the martingale moment, or by a squared TD loss together with the normalization constraint of Proposition~\ref{prop:q_properties} or Proposition~\ref{prop:q_star_normalization}.

\subsection{Off-Policy Learning of Frozen Response Maps}\label{subsec:offpolicy}

To learn the dependence on $u^{-i}$, choose a design distribution $\nu_i$ on $U^{N-1}$.  At the beginning of each trajectory, draw $\bar u^{-i}\sim\nu_i$ and keep it fixed on the whole trajectory.  Player $i$ is sampled from a full-support behavior policy $\pi_b$ at the grid points.  Across trajectories, the regression variables $(t,x,u^i,\bar u^{-i})$ cover the desired region of the frozen response map.

For direct learning of the optimal pair $(\tilde V^i,q_*^i)$, use the residual
\begin{align}\label{eq:TD_offpolicy}
\delta_{k}^{i,\mathrm{off}}={}&e^{-\beta^i s_{k+1}}\hat J_*^i(s_{k+1},X_{s_{k+1}},\bar u^{-i})
-e^{-\beta^i s_k}\hat J_*^i(s_k,X_{s_k},\bar u^{-i})\\
&+\int_{s_k}^{s_{k+1}}e^{-\beta^i r}\big[f^i(r,X_r,a_r^{\mathcal G,\pi_b})
-\hat q_*^i(r,X_r,a_r^{\mathcal G,\pi_b},\bar u^{-i})\big]dr.
\end{align}
Theorem~\ref{thm martingale represent 3} identifies the optimal pair when these martingale moments hold for every initial condition and grid, together with the Gibbs normalization.

\begin{algorithm}[H]
\caption{Off-Policy Learning of Player $i$'s Frozen Response Map}
\label{alg:offpolicy_q}
\begin{algorithmic}[1]
\REQUIRE Design law $\nu_i$ for $u^{-i}$; full-support behavior policy $\pi_b$; time grid; learning rates
\STATE Initialize parameters $\theta^i,\phi^i$
\FOR{episode $e=1,2,\ldots$}
  \STATE Draw $\bar u^{-i}\sim\nu_i$ and keep $\bar u^{-i}$ fixed during the episode
  \FOR{grid cells $[s_k,s_{k+1}]$}
    \STATE Draw $a^i_{s_k}\sim\pi_b(\cdot;s_k,X_{s_k})$ and hold it on $[s_k,s_{k+1})$
    \STATE Evolve the state with action profile $(a^i_{s_k},\bar u^{-i})$
    \STATE Form $\delta_k^{i,\mathrm{off}}$ from~\eqref{eq:TD_offpolicy} and update $(\theta^i,\phi^i)$
  \ENDFOR
\ENDFOR
\RETURN $\hat q_*^{i,\theta^i}$ and $\hat J_*^{i,\phi^i}$
\end{algorithmic}
\end{algorithm}

Once a frozen $q^i(\cdot,u^{-i})$ has been identified, its Gibbs improvement is given by~\eqref{eq:improved_policy_q}.  Repeating the experiment over the design values of $u^{-i}$ estimates the family of frozen conditional response maps.  The algorithm identifies the conditional response maps of the model; it is not asserted to learn the standard Nash equilibrium.

\subsection{Enforcing the Normalization Constraint}\label{sec:normalization}

The optimal characterization requires
\[
\int_U\exp\{\hat q_*^i(t,x,u^i,u^{-i})/\gamma^i\}\,du^i=1.
\]
This condition is part of the HJB identification in Theorem~\ref{thm martingale represent 3}; a martingale loss alone does not determine the required normalization.  A direct parameterization is
\[
\hat q_*^{i,\theta}=\bar q_*^{i,\theta}
-\gamma^i\log\int_U\exp\{\bar q_*^{i,\theta}/\gamma^i\}\,du^i,
\]
which enforces the constraint identically whenever the integral is finite.  A penalty on the normalization residual can be used when this integral is evaluated numerically.

\section{Numerical Experiments}\label{sec:numerical}

The numerical example is a stationary discounted LQ benchmark. It serves two purposes. First, the compatibility criterion can be evaluated in closed form: in the asymmetric specification below, a time-local conditional equilibrium exists only at one ratio of the exploration weights, and the corresponding joint law is explicit. Second, a regression experiment estimates the action-dependent coefficients of the one-step TD score generated by the frozen assessment, illustrating how the conditional response maps can be identified from joint exploratory data.

\subsection{Two-Player Stationary Discounted LQ Benchmark}

Throughout this section we write $\rho>0$ for the common discount rate of both players (i.e.\ $\rho=\beta^1=\beta^2$ in the notation of Section~\ref{sec:ergodic}); this is independent of the exploration weights $\gamma^i>0$.

The one-dimensional state $X_t \in \mathbb{R}$ evolves as
\begin{align*}
dX_t = (A X_t + B^1 a_{1,t} + B^2 a_{2,t})\,dt + \sigma\,dW_t,
\end{align*}
where $A < 0$ ensures mean-reversion and $B^i$ measures player $i$'s influence on the state. Each player maximizes the infinite-horizon discounted entropy-regularized reward
\begin{align*}
V^i(\pi^1,\pi^2) = \mathbb{E}\bigg[\int_0^\infty e^{-\rho t}\big(-Q^i X_t^2 - R^i a_{i,t}^2 - \gamma^i\log\pi^i(a_{i,t}\mid X_t)\big)\,dt\bigg],
\end{align*}
i.e.\ a stationary discounted game with running reward $f^i(x,u^i)=-(Q^ix^2+R^i(u^i)^2)$ and entropy bonus $-\gamma^i\log\pi^i$.

The continuation assessment of Section~\ref{subsec:assessment} is available in closed form.  Holding the opponents at $a^{-i}$ over the remaining horizon gives
\begin{align}\label{eq:frozen_lq}
\tilde V^i(x;a^{-i})=-P^ix^2+\ell^i(a^{-i})\,x+m^i(a^{-i}),
\end{align}
in which $P^i$ is the positive root of the scalar Riccati equation
\begin{align}\label{eq:scalar_riccati}
\frac{(B^i)^2}{R^i}P^2+(\rho-2A)P-Q^i=0,
\end{align}
and
\begin{align}\label{eq:ell}
\ell^i(a^{-i})=\frac{-2P^i\sum_{j\ne i}B^ja^j}{E^i},\qquad
E^i:=\rho-A+\frac{(B^i)^2P^i}{R^i}.
\end{align}
The quadratic coefficient does not depend on $a^{-i}$, since a frozen opponent enters the dynamics only as a constant drift.  For the parameters used below,
\[
P^1=0.3022,\qquad P^2=0.8000 .
\]
These are not the coefficients of the coupled algebraic Riccati system, which gives $P^1=0.2746$ and $P^2=0.4426$; the two objects differ, and only the former enters the assessment of the present model.

The dependence on the opponent's action enters through the linear term in~\eqref{eq:frozen_lq}. Because $\ell^i$ depends on $a^{-i}$, the spatial derivative $\partial_x\tilde V^i=-2P^ix+\ell^i(a^{-i})$ varies with the opponents' actions, so $\partial_{a^j}\partial_x\tilde V^i$ does not vanish and the indirect term in the cross-partial decomposition is nonzero.

\subsection{Compatibility and the Joint Law}

Substituting \eqref{eq:frozen_lq}--\eqref{eq:ell} into the decomposition of Proposition~\ref{prop:decomposition}, the direct term vanishes because the drift is linear in the actions and the diffusion does not depend on them, leaving
\begin{align}\label{eq:mixed_lq}
\frac{\partial^2 H_v^i}{\partial a^j\partial a^i}
=B^i\frac{\partial\ell^i}{\partial a^j}
=-\frac{2P^iB^iB^j}{E^i},\qquad j\ne i .
\end{align}
The compatibility gap in~\eqref{eq:compat_gap} is therefore
\begin{align}\label{eq:gap_lq}
\Delta^{12}=-2B^1B^2\left(\frac{\varphi^1}{\gamma^1}-\frac{\varphi^2}{\gamma^2}\right),
\qquad \varphi^i:=\frac{P^i}{E^i},
\end{align}
and does not depend on $(t,x)$.  With the parameters above, $\varphi^1=0.1309$ and $\varphi^2=0.5333$, so $\Delta^{12}=0$ exactly when
\begin{align}\label{eq:rstar}
\frac{\gamma^2}{\gamma^1}=\frac{\varphi^2}{\varphi^1}=4.0745 .
\end{align}
At equal exploration weights and $\gamma^1=1$ the gap equals $+0.80$; it decays like $O(\gamma^{-1})$, as in Proposition~\ref{prop:large_gamma_approx}, and decreases monotonically in $\gamma^2/\gamma^1$, changing sign at~\eqref{eq:rstar} (Figure~\ref{fig:compatibility}, left and Table~\ref{tab:gap_scan}).

\begin{table}[ht]
\centering
\caption{Compatibility gap for the LQ benchmark with $\gamma^1=1$.}
\label{tab:gap_scan}
\small
\begin{tabular}{lccccccc}
\toprule
$\gamma^2/\gamma^1$ & $1$ & $2$ & $4$ & $\mathbf{4.0745}$ & $5$ & $8$ & $10$ \\
\midrule
$\Delta^{12}$ & $0.8049$ & $0.2715$ & $0.0049$ & $\mathbf{0}$ & $-0.0485$ & $-0.1285$ & $-0.1551$ \\
\bottomrule
\end{tabular}
\end{table}

At the ratio~\eqref{eq:rstar} the joint law exists and is Gaussian.  Writing $\Phi=\log\psi$, the cross-partial identities give the cross coefficient $c=B^1\partial_{a^2}\ell^1/\gamma^1=-0.2618$, so that
\[
\Phi=-\frac{R^1}{\gamma^1}(a^1)^2-\frac{R^2}{\gamma^2}(a^2)^2+c\,a^1a^2+\text{(terms linear in the actions and depending on }x).
\]
Thus $\psi$ is the bivariate Gaussian with covariance $\Sigma=\tfrac12Q^{-1}$, where $Q=\begin{pmatrix}R^1/\gamma^1&-c/2\\-c/2&R^2/\gamma^2\end{pmatrix}$.  With $\gamma^1=1$ and $\gamma^2=4.0745$ one has $\det Q=0.1056>0$ and
\[
\Sigma=\begin{pmatrix}0.5811&-0.6199\\-0.6199&4.7357\end{pmatrix}.
\]
Its full conditionals reproduce the two Gibbs responses exactly: the conditional variances are $0.5000=\gamma^1/(2R^1)$ and $4.0745=\gamma^2/(2R^2)$ (Figure~\ref{fig:compatibility}, right).

The criterion is independent of $x$, so it holds on the whole state space or nowhere.  It is restrictive here because the players are asymmetric.  If $B^1=B^2$, $Q^1=Q^2$ and $R^1=R^2$, then $\varphi^1=\varphi^2$ and equal exploration weights are compatible; asymmetry between the players has to be matched by asymmetry in the exploration weights.

\begin{figure}[ht]
\centering
\includegraphics[width=\textwidth]{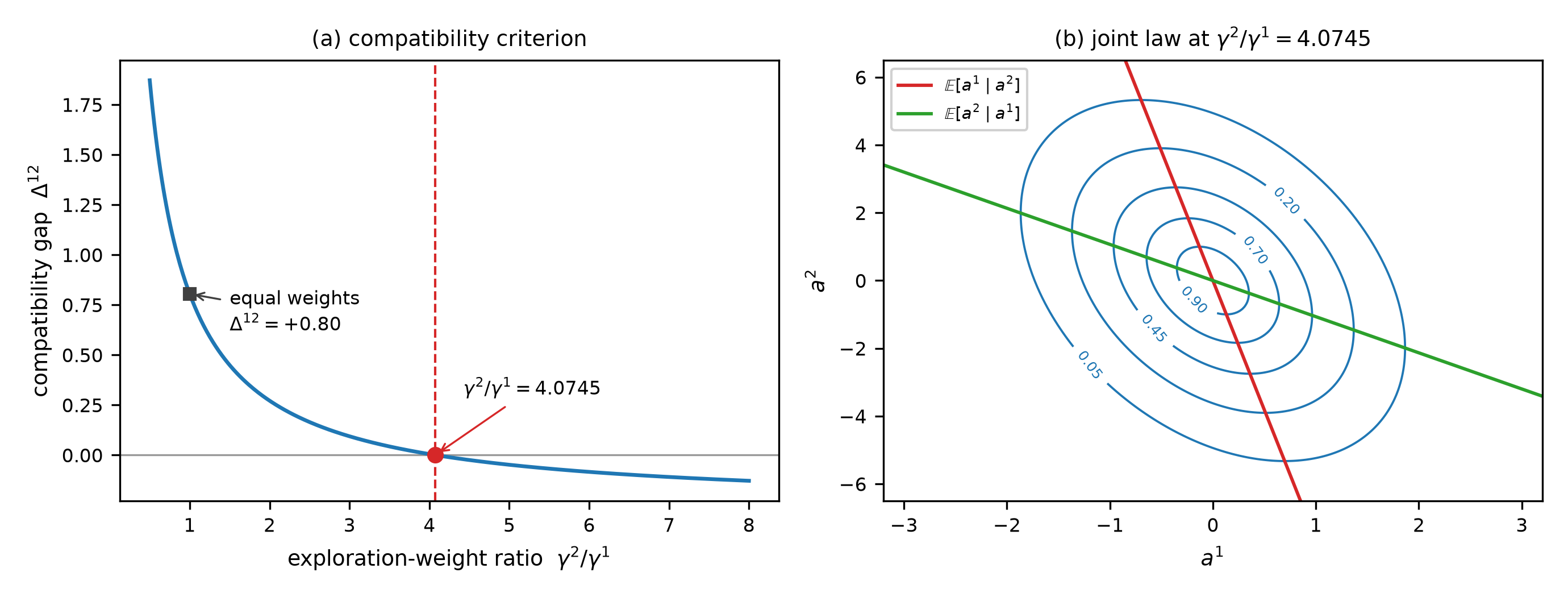}
\caption{Compatibility for the LQ benchmark.  \emph{Left:} the compatibility gap $\Delta^{12}$ of~\eqref{eq:gap_lq} as a function of the exploration-weight ratio $\gamma^2/\gamma^1$; it is positive at equal weights and crosses zero at~\eqref{eq:rstar}.  \emph{Right:} level sets of the joint law $\psi$ at that ratio, with the two conditional means; the conditional variances agree with $\gamma^i/(2R^i)$ to machine precision.}
\label{fig:compatibility}
\end{figure}

\subsection{One-Step Action-Score Parametrization}

The regression experiment fits the action-dependent part of the one-step score generated by the frozen assessment of Section~\ref{subsec:assessment}, and asks whether its coefficients can be recovered from joint exploratory data.  With $\tilde V^i$ as in~\eqref{eq:frozen_lq}--\eqref{eq:ell} and $\ell'^i:=\partial_{a^{-i}}\ell^i$, the action-dependent part of the Hamiltonian is
\begin{align}\label{eq:q_theoretical_num}
h_i^{\tilde V}(x,a^i,a^{-i})
={}&-R^i(a^i)^2-2P^iB^ia^ix-2P^iB^{-i}a^{-i}x\nonumber\\
&+\ell'^iB^ia^ia^{-i}+\ell'^iB^{-i}(a^{-i})^2+C^i_0(x),
\end{align}
where $C^i_0(x)$ is independent of the actions.  The two terms on the second line are absent whenever the assessment does not depend on the opponents' actions, as in the Nash benchmark, and the cross term $a^ia^{-i}$ is the one that carries the compatibility content of the model.

The experiment uses a time step $\Delta t$ and the one-step increment
\[
Y^i_\Delta=f^i(x,a^i)\Delta t+\delta_\rho\tilde V^i(X_{t+\Delta t};a^{-i})-\tilde V^i(x;a^{-i}),
\qquad \delta_\rho=e^{-\rho\Delta t}.
\]
Because $\tilde V^i$ depends on $a^{-i}$, the last term is not action-independent, which is the second respect in which the present assessment differs from one that ignores the opponents' actions.  Expanding one Euler step conditional on $(x,a^i,a^{-i})$ gives
\begin{align}\label{eq:q_conditional_param}
\mathbb E[Y^i_\Delta\mid x,a]
={}&-\frac{e^{-\psi^i_3}}{2}(a^i-\psi^i_1x-\psi^i_2)^2
+\eta^i_1a^{-i}x+\eta^i_2a^{-i}\nonumber\\
&+\eta^i_3a^ia^{-i}+\eta^i_4(a^{-i})^2+C^i_\Delta(x)+O(\Delta t^2),
\end{align}
up to an action-independent baseline and $O(\Delta t^2)$ terms.  Every coefficient below is fitted; two of the regressors, $a^ia^{-i}$ and $(a^{-i})^2$, appear only because the assessment depends on the opponents' actions, and dropping them would leave the fit misspecified exactly in the direction that determines compatibility.  The coefficients are
\begin{align}\label{eq:true_params}
\psi^i_1&=\delta_\rho K^i, \qquad \psi^i_2=0, \qquad e^{-\psi^i_3}=2R^i\Delta t,\nonumber\\
\eta^i_1&=(\delta_\rho-1)\ell'^i+\delta_\rho\Delta t\big(A\ell'^i-2P^iB^{-i}\big),\qquad \eta^i_2=0,\nonumber\\
\eta^i_3&=\delta_\rho\Delta t\,\ell'^iB^i, \qquad
\eta^i_4=(\delta_\rho-1)m^i_{(2)}+\delta_\rho\Delta t\,\ell'^iB^{-i},
\end{align}
with $K^i=-P^iB^i/R^i$ and $m^i_{(2)}$ the coefficient of $(a^{-i})^2$ in the constant term of~\eqref{eq:frozen_lq}.  The fitted quantities are therefore coefficients of the one-step action score, not of the normalised continuous-time $q^i_*$; dividing the action-dependent part by $\Delta t$ recovers the Hamiltonian rate to first order.  If the one-step score is used in a Gibbs update, the corresponding one-step entropy weight is $\gamma^i\Delta t$, so the limiting Gaussian variance is $\gamma^i/(2R^i)$ rather than $\gamma^ie^{\psi^i_3}$.

\subsection{Experimental Design}

We choose an asymmetric game to make the opponent-action coupling parameters numerically distinguishable between the two players:

\begin{center}
\begin{tabular}{lll}
\toprule
\textbf{Parameter} & \textbf{Value} & \textbf{Role} \\
\midrule
$A$          & $-1.0$  & Mean-reversion \\
$B^1, B^2$   & $2.0, 0.5$ & Player~1 has $4\times$ stronger control \\
$Q^1, Q^2$   & $1.0, 2.0$ & Asymmetric state penalties \\
$R^1, R^2$   & $1.0, 0.5$ & Asymmetric control costs \\
$\sigma$     & $0.3$   & Diffusion coefficient \\
$\rho$       & $0.1$   & Discount rate \\
$\Delta t$   & $0.02$  & Discretization step \\
\bottomrule
\end{tabular}
\end{center}

The $B^1 \gg B^2$ asymmetry makes $\eta^2_1$ (Player~2's sensitivity to Player~1's actions) much larger than $\eta^1_1$.  The scalar Riccati equation~\eqref{eq:scalar_riccati} gives
\[
P^1 = 0.3022, \quad P^2 = 0.8000, \qquad K^1 = -0.6044, \quad K^2 = -0.8000,
\]
and the coefficients~\eqref{eq:true_params} evaluated at these values are:

\begin{center}
\begin{tabular}{lcccc}
\toprule
\textbf{Player} & $\psi_1$ & $\eta_1$ & $\eta_3$ & $\eta_4$ \\
\midrule
P1 & $-0.6032$ & $-0.003158$ & $-0.005225$ & $-0.000738$ \\
P2 & $-0.7984$ & $-0.017028$ & $-0.021291$ & $+0.039868$ \\
\bottomrule
\end{tabular}
\end{center}

The two contributions to $\eta^i_1$ in~\eqref{eq:true_params} nearly cancel for Player~1, which is why that coefficient is an order of magnitude smaller than the cross coefficient $\eta^i_3$.

We generate joint exploratory data with independent behaviour actions $a^1,a^2\sim\mathcal N(0,1)$ and fit the action-dependent coefficients in~\eqref{eq:q_conditional_param}; an $x$-only baseline absorbs the remaining terms.  The assessment $\tilde V^i(\cdot;a^{-i})$ is held at its closed form~\eqref{eq:frozen_lq}.  The discount rate is $\rho$; the exploration weights $\gamma^i$ do not enter this regression, since the fitted object is the one-step score rather than its Gibbs normalisation.  We use $3\times10^6$ steps.  The $\psi$-block is trained with Adam and the $\eta$-block with annealed stochastic gradient descent, $\alpha_\eta$ decreasing from $4\times10^{-4}$ to $10^{-4}$ over the final third.  The two parameter blocks have different gradient scales. Since $\partial q/\partial\psi^i_1\propto e^{-\psi^i_3}=2R^i\Delta t\approx0.04$, a common step size makes the feedback coefficients converge about $25$ times more slowly than the $\eta$-coefficients. We therefore use Adam for the $\psi$-block and annealing for the smaller $\eta$-coefficients.

\subsection{Results}

Figure~\ref{fig:convergence} shows the fitted action-score parameters over $3\times10^6$ off-policy TD steps.  The feedback coefficients $\psi^i_1$, the curvature parameters $\psi^i_3$, and the two coefficients $\eta^i_3$ and $\eta^i_4$ that exist only under the frozen assessment all settle at their closed-form values, while the intercepts $\psi^i_2$ and the direct couplings $\eta^i_2$ stay near zero.  Figure~\ref{fig:eta1_convergence} enlarges the two opponent coefficients, and Table~\ref{tab:eta_results} records the final estimates.

\begin{figure}[ht]
\centering
\includegraphics[width=0.56\textwidth]{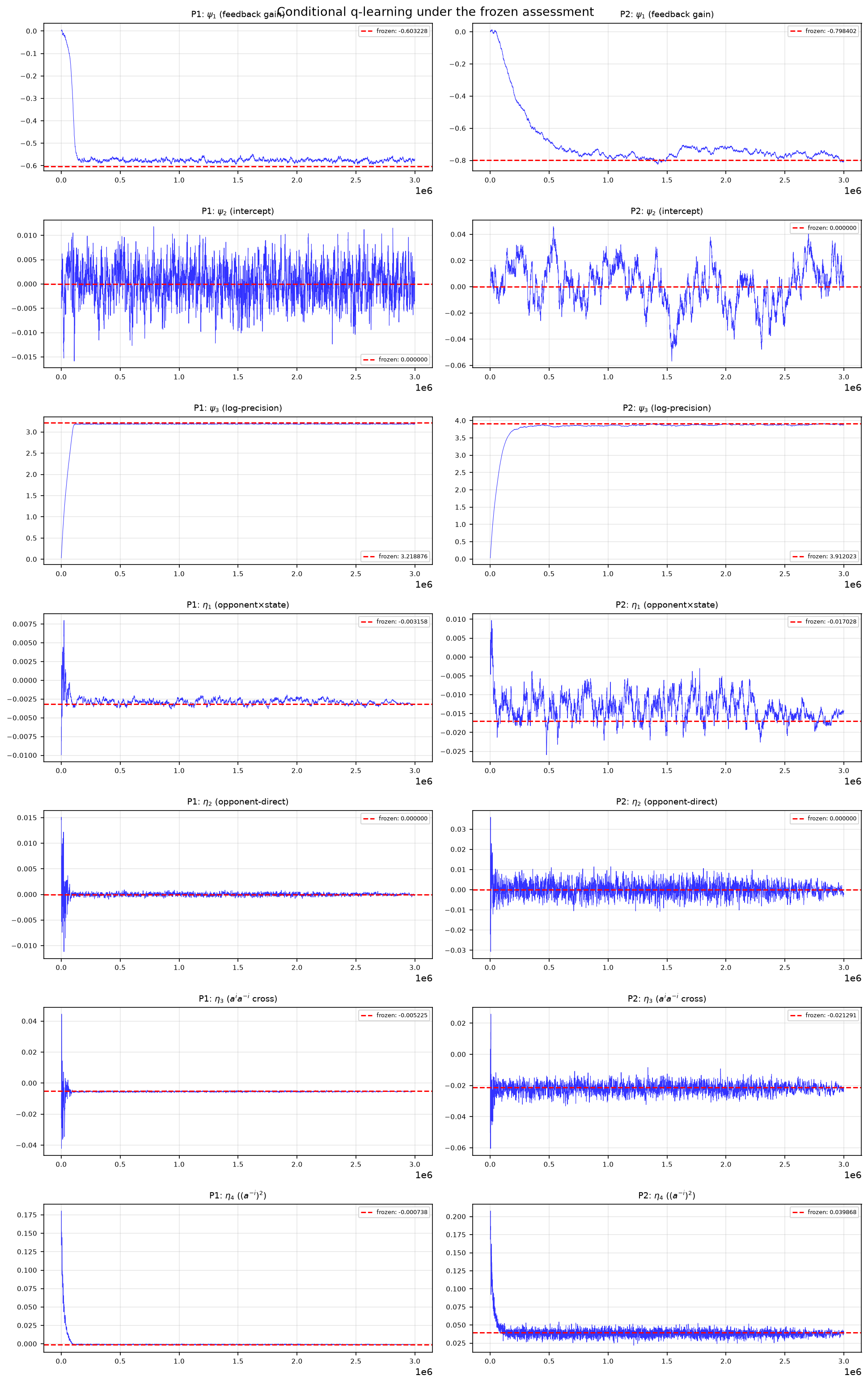}
\caption{Convergence of the fitted action-score parameters under the frozen assessment over $3\times10^6$ off-policy TD steps.  Red dashed: closed-form values from~\eqref{eq:true_params}.}
\label{fig:convergence}
\end{figure}

\begin{figure}[ht]
\centering
\includegraphics[width=0.85\textwidth]{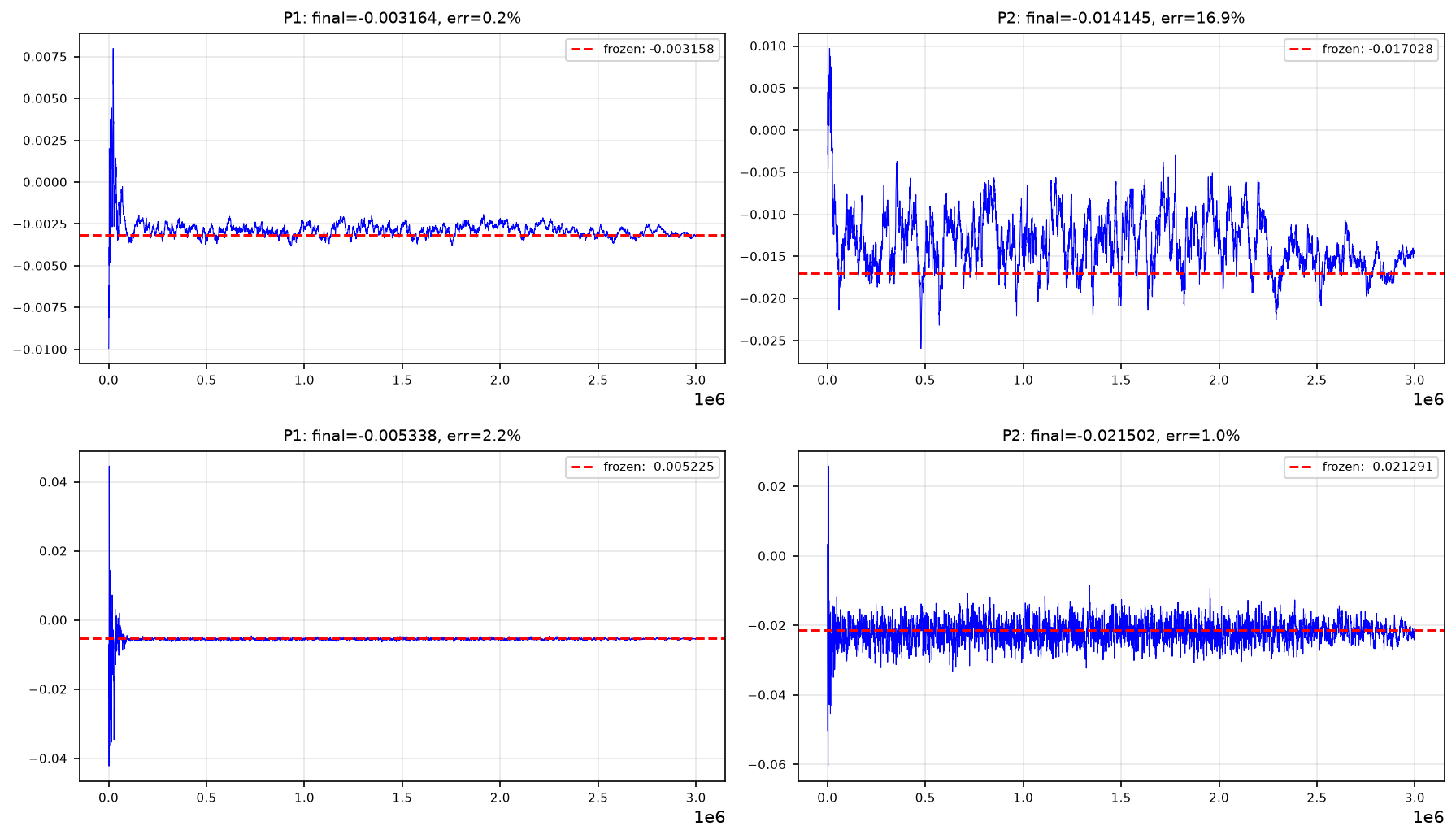}
\caption{Enlarged view of $\eta_1$ (top) and of the cross coefficient $\eta_3$ (bottom).  The cross coefficient enters only because the assessment depends on the opponents' actions, and it is recovered to within $2.2\%$ for Player~1 and $1.0\%$ for Player~2.}
\label{fig:eta1_convergence}
\end{figure}

\begin{table}[ht]
\centering
\caption{One-step action-score fit: final parameter estimates ($3\times10^6$ steps).}
\label{tab:eta_results}
\begin{tabular}{lccc}
\toprule
\textbf{Parameter} & \textbf{Closed form} & \textbf{Learned} & \textbf{Rel.\ error} \\
\midrule
\multicolumn{4}{c}{\textbf{Player 1}} \\
\midrule
$\psi^1_1$ (feedback)              & $-0.603228$ & $-0.570967$ & $5.3\%$ \\
$\eta^1_1$ (opponent$\times$state) & $-0.003158$ & $-0.003164$ & $0.2\%$ \\
$\eta^1_3$ (cross)                 & $-0.005225$ & $-0.005338$ & $2.2\%$ \\
$\eta^1_4$ (opponent squared)      & $-0.000738$ & $-0.000630$ & $14.7\%$ \\
\addlinespace
\multicolumn{4}{c}{\textbf{Player 2}} \\
\midrule
$\psi^2_1$ (feedback)              & $-0.798402$ & $-0.805611$ & $0.9\%$ \\
$\eta^2_1$ (opponent$\times$state) & $-0.017028$ & $-0.014145$ & $16.9\%$ \\
$\eta^2_3$ (cross)                 & $-0.021291$ & $-0.021502$ & $1.0\%$ \\
$\eta^2_4$ (opponent squared)      & $+0.039868$ & $+0.037661$ & $5.5\%$ \\
\bottomrule
\end{tabular}
\end{table}

The fitted coefficients are close to their closed-form values. The cross coefficient $\eta^i_3$ is recovered to within $2.2\%$ and $1.0\%$, and the feedback coefficients to within $5.3\%$ and $0.9\%$.  The largest relative error, $14.7\%$, falls on $\eta^1_4$, whose closed-form value is a near-cancellation and therefore small in absolute terms.

The learned feedback gains $K^i_L=\psi^i_1/\delta_\rho$ are $-0.578$ and $-0.769$, against the frozen response gains $K^i=-P^iB^i/R^i=-0.604$ and $-0.800$.  Simulating the closed loop under each for $5\times10^4$ steps after a $5\times10^3$-step burn-in gives the trajectories and stationary densities of Figure~\ref{fig:ergodic}; the empirical standard deviations are $0.1366$ and $0.1327$, against the closed-form value $0.1313$ for the frozen gain.

\begin{figure}[ht]
\centering
\includegraphics[width=\textwidth]{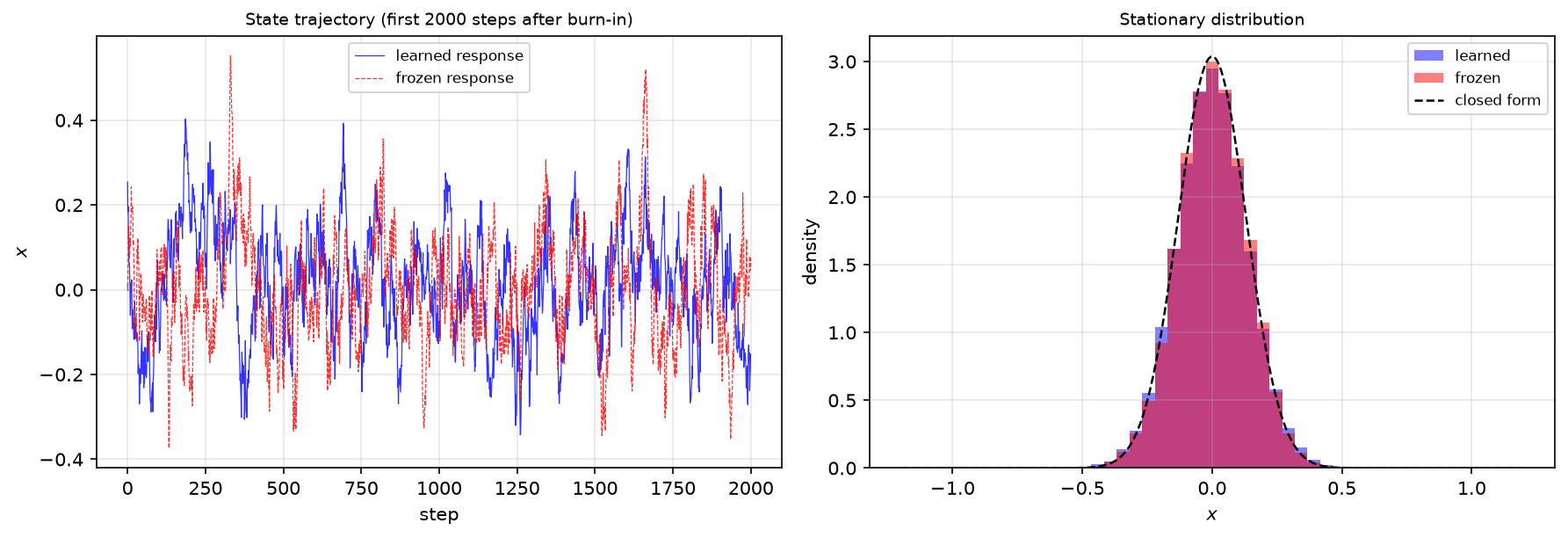}
\caption{Closed-loop behaviour under the learned feedback (blue) and under the frozen response (red).  \emph{Left:} state trajectories.  \emph{Right:} empirical stationary densities, with the closed-form density of the frozen response as the dashed curve.}
\label{fig:ergodic}
\end{figure}

\subsection{Discussion}

The regression identifies every coefficient of the one-step score, including the two regressors that exist only under the frozen assessment.  Its role is identification; compatibility is not tested by this experiment.  The criterion is evaluated in closed form above, where it is a single scalar condition that equal exploration weights fail to meet for these parameters.

\section{Concluding Remarks}\label{sec:concluding}

The model studied here is a time-local conditional equilibrium.  Each player responds to the actions currently realized by the others, and because no player can forecast the opponents' future actions, the future is assessed with the opponents held at their realized values.  For a fixed assessment, the current response is a Gibbs density conditional on the realized opponent profile, and an equilibrium is a joint law whose full conditionals agree with these responses.  Existence of such a joint law is equivalent to compatibility of the full conditionals.  The standard exploratory Nash problem is a fixed point in complete continuation policies and is a distinct object.

For smooth positive Gibbs responses, compatibility is characterized by the cross-partial condition in Theorem~\ref{thm:compat_q}; on simply connected action domains it is equivalent to an entropy-weighted potential structure and gives an explicit Gibbs joint density.  In stationary discounted settings this amounts to statewise consistency of current conditional behavior rather than consistency of complete continuation rules, while the evolution of the state remains in the continuation value.  Permutation symmetry alone is not sufficient.

The short-sightedness is the defining feature of the model and also its main limitation.  Two extensions would remove it.  The first is to make the assessment self-consistent: rather than holding the opponents at their realized actions, require the continuation value to be the value generated by the dynamics induced by the equilibrium joint law.  The construction then becomes a two-level fixed point, and existence has to be established rather than assumed.  The second is to return to complete continuation policies and solve the standard Nash problem directly, which requires a coupled system of fully nonlinear HJB equations together with a fixed point in policy space.  Conditioning on complete continuation controls rather than a single time slice leads to a path- or strategy-space compatibility problem, in which nonanticipative conditional laws replace the finite-dimensional Poincar\'e argument used here.

\section*{Declarations}

\begin{itemize}
\item \textbf{Conflict of interest.} The author declares no conflict of interest.
\item \textbf{Ethics approval and consent to participate.} Not applicable.
\item \textbf{Consent for publication.} Not applicable.
\item \textbf{Data availability.} The numerical experiments of Section~\ref{sec:numerical} use synthetic data generated by the algorithms described there; no external datasets were analysed. The generating scripts are available from the author on reasonable request.
\item \textbf{Code availability.} The scripts implementing the numerical experiments of Section~\ref{sec:numerical} are available from the author on reasonable request.
\item \textbf{Author contributions.} The author is solely responsible for the conception of the study, the mathematical analysis, and the writing of the manuscript.
\item \textbf{Use of large language models.} During the preparation of this work the author used DeepSeek and ChatGPT to assist with writing and with reasoning checks. The model, mathematical ideas, and methods presented here are the author's own. The author reviewed and verified all content and takes full responsibility for the manuscript.
\end{itemize}

\begin{appendices}

\section{Proof of the Dynamic Programming Principle}\label{app:DPP}

Fix $\pi^{-i}\in\Pi^{N-1}$.  For this fixed opponent profile the exploratory best-response problem is an ordinary Markov stochastic control problem.  The averaged state equation is well posed by the estimates preceding Proposition~\ref{prop:DPP}.  The running reward
\[
\tilde f^i(s,X_s,\pi_s^i)-\gamma^i\mathcal E(s,X_s,\pi_s^i)
\]
is measurable, additive over time intervals, and integrable under Definition~\ref{def:admissible_policy}.  Under the stability hypothesis in Proposition~\ref{prop:DPP}, the admissible controls satisfy the concatenation requirement of the weak dynamic programming principle.  Thus Theorem~3.1 of \cite{BouchardTouzi11} applies after augmenting the state by the accumulated discounted running reward.

Let $(\tilde V^i)^*$ be the upper semicontinuous envelope of $\tilde V^i(\cdot,\cdot;\pi^{-i})$.  The upper weak-DPP inequality gives
\begin{align*}
\tilde V^i(t,x;\pi^{-i})
\le \sup_{\pi^i\in\Pi}\mathbb E\bigg[&\int_t^\tau e^{-\beta^i(r-t)}
\big(\tilde f^i(r,X_r,\pi_r^i)-\gamma^i\mathcal E(r,X_r,\pi_r^i)\big)dr\\
&+e^{-\beta^i(\tau-t)}(\tilde V^i)^*(\tau,X_\tau;\pi^{-i})\bigg].
\end{align*}
The lower weak-DPP inequality of \cite[Thm.~3.1]{BouchardTouzi11} holds with any upper-semicontinuous function $\varphi\le\tilde V^i$ in place of the continuation value.  Under the continuity hypothesis of Proposition~\ref{prop:DPP}, $(\tilde V^i)^*=\tilde V^i$ and we may take $\varphi=\tilde V^i$.  The two inequalities then give~\eqref{eq:DPP_stopping}.  Taking $\tau\equiv s$ gives~\eqref{eq:DPP_deterministic}.

This use of the weak DPP avoids assuming the existence of an $\mathscr F_\tau$-measurable selector of pointwise $\varepsilon$-optimal Markov policies.  Such a selector was not part of Definition~\ref{def:admissible_policy} and is not needed for the HJB derivation.

\end{appendices}

\bibliographystyle{plain}
\bibliography{references}

@book{Krylov87,
  author    = {N.V. Krylov},
  title     = {Nonlinear Elliptic and Parabolic Equations of the Second Order},
  publisher = {D.\ Reidel Publishing Company},
  address   = {Dordrecht},
  year      = {1987}
}

@article{BouchardTouzi11,
  author  = {B. Bouchard and N. Touzi},
  title   = {Weak dynamic programming principle for viscosity solutions},
  journal = {SIAM Journal on Control and Optimization},
  volume  = {49},
  number  = {3},
  pages   = {948--962},
  year    = {2011},
  doi     = {10.1137/090752328}
}

@book{FlemSoner,
  author    = {W.H. Fleming and H.M. Soner},
  title     = {Controlled {M}arkov Processes and Viscosity Solutions},
  edition   = {2nd},
  publisher = {Springer},
  address   = {New York},
  year      = {2006}
}

@book{GT01,
  author    = {D. Gilbarg and N.S. Trudinger},
  title     = {Elliptic Partial Differential Equations of Second Order},
  edition   = {2nd},
  publisher = {Springer},
  address   = {Berlin},
  year      = {2001}
}

@article{JZ22a,
  author  = {Jia, Yanwei and Zhou, Xun Yu},
  title   = {Policy Evaluation and Temporal-Difference Learning in Continuous Time and Space: A Martingale Approach},
  journal = {Journal of Machine Learning Research},
  volume  = {23},
  number  = {154},
  pages   = {1--55},
  year    = {2022}
}

@article{JZ22b,
  author  = {Jia, Yanwei and Zhou, Xun Yu},
  title   = {Policy Gradient and Actor-Critic Learning in Continuous Time and Space: Theory and Algorithms},
  journal = {Journal of Machine Learning Research},
  volume  = {23},
  number  = {275},
  pages   = {1--50},
  year    = {2022}
}

@article{JZ23,
  author  = {Y. Jia and X.Y. Zhou},
  title   = {{$q$}-{L}earning in continuous time},
  journal = {Journal of Machine Learning Research},
  volume  = {24},
  number  = {161},
  pages   = {1--61},
  year    = {2023}
}

@article{WZZ20,
  author  = {H. Wang and T. Zariphopoulou and X.Y. Zhou},
  title   = {Reinforcement learning in continuous time and space: {A} stochastic control approach},
  journal = {Journal of Machine Learning Research},
  volume  = {21},
  number  = {198},
  pages   = {1--34},
  year    = {2020}
}

@article{TangZhou24,
  author = {Tang, Wenpin and Zhou, Xun Yu},
  title = {Regret of Exploratory Policy Improvement and q-Learning},
  journal = {arXiv:2411.01302},
  year = {2024}
}

@article{HuangJZ24,
  author = {Huang, Yilie and Jia, Yanwei and Zhou, Xun Yu},
  title = {Sublinear Regret for a Class of Continuous-Time Linear-Quadratic Reinforcement Learning Problems},
  journal = {SIAM Journal on Control and Optimization},
  volume = {63},
  number = {5},
  pages = {3452--3474},
  year = {2025}
}

@article{HuangZhou25,
  author = {Huang, Yilie and Zhou, Xun Yu},
  title = {Data-Driven Exploration for a Class of Continuous-Time Indefinite Linear-Quadratic Reinforcement Learning Problems},
  journal = {arXiv:2507.00358},
  year = {2025}
}

@article{MondererShapley96,
  author = {Monderer, Dov and Shapley, Lloyd S.},
  title = {Potential Games},
  journal = {Games and Economic Behavior},
  volume = {14},
  number = {1},
  pages = {124--143},
  year = {1996},
  doi = {10.1006/game.1996.0044}
}

@article{Blume93,
  author = {Blume, Lawrence E.},
  title = {The Statistical Mechanics of Strategic Interaction},
  journal = {Games and Economic Behavior},
  volume = {5},
  number = {3},
  pages = {387--424},
  year = {1993},
  doi = {10.1006/game.1993.1023}
}

@article{BlumeDurlauf03,
  author = {Blume, Lawrence E. and Durlauf, Steven N.},
  title = {Equilibrium Concepts for Social Interaction Models},
  journal = {International Game Theory Review},
  volume = {5},
  number = {3},
  pages = {193--209},
  year = {2003},
  doi = {10.1142/S021919890300101X}
}

@article{ArnoldPress89,
  author  = {Arnold, Barry C. and Press, S. James},
  title   = {Compatible Conditional Distributions},
  journal = {Journal of the American Statistical Association},
  volume  = {84},
  number  = {405},
  pages   = {152--156},
  year    = {1989},
  doi     = {10.1080/01621459.1989.10478750}
}

@book{ArnoldCastilloSarabia92,
  author    = {Arnold, Barry C. and Castillo, Enrique and Sarabia, Jos{\'e}-Mar{\'i}a},
  title     = {Conditionally Specified Distributions},
  series    = {Lecture Notes in Statistics},
  volume    = {73},
  publisher = {Springer},
  address   = {New York},
  year      = {1992},
  doi       = {10.1007/978-1-4612-2912-4}
}

@article{JZ23err,
  author  = {Jia, Yanwei and Zhou, Xun Yu},
  title   = {Erratum to ``q-Learning in Continuous Time''},
  journal = {Journal of Machine Learning Research},
  year    = {2025}
}

@article{Aumann74,
  author  = {Aumann, Robert J.},
  title   = {Subjectivity and Correlation in Randomized Strategies},
  journal = {Journal of Mathematical Economics},
  volume  = {1},
  number  = {1},
  pages   = {67--96},
  year    = {1974},
  doi     = {10.1016/0304-4068(74)90037-8}
}

@inproceedings{OrtizSchapireKakade07,
  author    = {Ortiz, Luis E. and Schapire, Robert E. and Kakade, Sham M.},
  title     = {Maximum Entropy Correlated Equilibria},
  booktitle = {Proceedings of the Eleventh International Conference on Artificial Intelligence and Statistics},
  series    = {Proceedings of Machine Learning Research},
  volume    = {2},
  pages     = {347--354},
  year      = {2007},
  publisher = {PMLR}
}

@inproceedings{CernyAnZhang22,
  author    = {{\v C}ern{\'y}, Jakub and An, Bo and Zhang, Allan N.},
  title     = {Quantal Correlated Equilibrium in Normal Form Games},
  booktitle = {Proceedings of the 23rd ACM Conference on Economics and Computation},
  pages     = {210--239},
  year      = {2022},
  publisher = {ACM},
  doi       = {10.1145/3490486.3538350}
}

@inproceedings{LiYuDorflerLygeros24,
  author    = {Li, Sarah H. Q. and Yu, Yue and D{\"o}rfler, Florian and Lygeros, John},
  title     = {A Coupled Optimization Framework for Correlated Equilibria in Normal-Form Games},
  booktitle = {2024 IEEE 63rd Conference on Decision and Control (CDC)},
  pages     = {1739--1744},
  year      = {2024},
  publisher = {IEEE},
  doi       = {10.1109/CDC56724.2024.10886643}
}

@article{Chen10,
  author  = {Chen, Hua Yun},
  title   = {Compatibility of Conditionally Specified Models},
  journal = {Statistics \& Probability Letters},
  volume  = {80},
  number  = {7--8},
  pages   = {670--677},
  year    = {2010},
  doi     = {10.1016/j.spl.2009.12.025}
}

@article{FirooziJaimungal22,
  author  = {Firoozi, Dena and Jaimungal, Sebastian},
  title   = {Exploratory {LQG} Mean Field Games with Entropy Regularization},
  journal = {Automatica},
  volume  = {139},
  pages   = {110177},
  year    = {2022},
  doi     = {10.1016/j.automatica.2022.110177}
}

@article{SunJia23,
  author  = {Sun, Zhongshi and Jia, Guangyan},
  title   = {Reinforcement Learning for Exploratory Linear-Quadratic Two-Person Zero-Sum Stochastic Differential Games},
  journal = {Applied Mathematics and Computation},
  volume  = {442},
  pages   = {127763},
  year    = {2023},
  doi     = {10.1016/j.amc.2022.127763}
}

@article{BenderThuan26,
  author  = {Bender, Christian and Thuan, Nguyen Tran},
  title   = {Continuous Time Reinforcement Learning: A Random Measure Approach},
  journal = {Stochastic Processes and their Applications},
  volume  = {194},
  pages   = {104848},
  year    = {2026},
  doi     = {10.1016/j.spa.2025.104848}
}

@article{GaoZhou26,
  author  = {Gao, Xuefeng and Zhou, Xunyu},
  title   = {Square-Root Regret Bounds for Continuous-Time Episodic Markov Decision Processes},
  journal = {Mathematics of Operations Research},
  volume  = {51},
  number  = {1},
  pages   = {333--357},
  year    = {2026},
  doi     = {10.1287/moor.2022.0283}
}

@article{MertikopoulosSandholm16,
  author  = {Mertikopoulos, Panayotis and Sandholm, William H.},
  title   = {Learning in Games via Reinforcement and Regularization},
  journal = {Mathematics of Operations Research},
  volume  = {41},
  number  = {4},
  pages   = {1297--1324},
  year    = {2016},
  doi     = {10.1287/moor.2016.0778}
}

@article{GuoXuZariphopoulou22,
  author  = {Guo, Xin and Xu, Renyuan and Zariphopoulou, Thaleia},
  title   = {Entropy Regularization for Mean Field Games with Learning},
  journal = {Mathematics of Operations Research},
  volume  = {47},
  number  = {4},
  pages   = {3239--3260},
  year    = {2022},
  doi     = {10.1287/moor.2021.1238}
}

@inproceedings{SavasAhmadiTanakaTopcu19,
  author    = {Savas, Yagiz and Ahmadi, Mohamadreza and Tanaka, Takashi and Topcu, Ufuk},
  title     = {Entropy-Regularized Stochastic Games},
  booktitle = {2019 IEEE 58th Conference on Decision and Control (CDC)},
  pages     = {5955--5962},
  year      = {2019},
  doi       = {10.1109/CDC40024.2019.9029555}
}

\end{document}